\documentclass[12pt]{article}
\pdfoutput=1	

\usepackage[
        backend=biber,
	giveninits,
        citestyle=alphabetic,
        bibstyle=alphabetic,
        sorting=nyvt,
        maxnames=5,
        maxalphanames=5,
        url=true,backref=false,isbn=false
           ]{biblatex}
\renewbibmacro*{doi+eprint+url}{%
  \iftoggle{bbx:doi}
    {\printfield{doi}}
    {}%
\newunit\newblock
  \iftoggle{bbx:eprint}
    {\usebibmacro{eprint}}
    {}%
\newunit\newblock
  \iftoggle{bbx:url}
        {%
      \iffieldundef{doi}
{\usebibmacro{url+urldate}}
        {%
          \clearfield{url}%
        }%
    }
    {}} 

\renewbibmacro*{in:}{%
  \iffieldundef{journaltitle}%
     {}%
     {\printtext{%
      \bibstring{in}\intitlepunct}}%
}

\usepackage[hidelinks,hypertexnames=false]{hyperref}
\usepackage[a4paper,hmargin=1.5cm,vmargin=2cm]{geometry}
\usepackage{xcolor}
\usepackage{amsmath}
\usepackage{amsfonts}
\usepackage{amsthm}
\usepackage{bbm}
\usepackage{tikz}
\def\indic#1{\mathbbm1_{\{#1\}}}
\def\E{\mathbb E}

\def\mP{\mathcal P}
\def\mO{\mathcal O}
\def\mA{{\mathcal A}}
\def\diffd{\text d}

\def\R{\mathbb{R}}
\let\Re\relax
\DeclareMathOperator\Re{Re}
\def\C{\mathbb{C}}
\def\diffd{\mathrm{d}}
\newtheorem{corr}{Corrolary}
\newtheorem{lem}{Lemma}
\newtheorem{prop}{Proposition}
\newtheorem{thm}{Theorem}

\title{Ahead of the Fisher-KPP front}
\author{Éric
Brunet\footnote{\href{mailto:eric.brunet@ens.fr}{eric.brunet@ens.fr}}\\\footnotesize Sorbonne Université,  Laboratoire de
Physique de l’École Normale Supérieure,\\[-1ex]\footnotesize ENS, Université PSL,
CNRS, Université Paris Cité, F-75005 Paris, France
}
\date{2023-02-20}
\begin{document}
\maketitle
\begin{abstract}
The solution $h$ to the  Fisher-KPP equation with a steep enough initial
condition develops into a front moving at
velocity~2, with logarithmic corrections to its position. In this paper we
investigate the value $h(ct,t)$ of the solution ahead of the front, at time~$t$
and position $ct$, with $c>2$. That value goes
to zero exponentially fast with time, with a well-known rate, but the
prefactor depends in
a non-trivial way of $c$, the initial condition and the non-linearity in
the equation. We compute an asymptotic expansion of that prefactor
for velocities $c$ close to 2. The expansion is surprisingly explicit and
irregular. The main tool of this paper is the so-called ``magical
expression'' which relates the position of the front, the initial
condition, and the quantity we investigate.
\end{abstract}
\section{Introduction}
The study of fronts interpolating from a stable solution to an unstable
solution is an important problem in mathematics, physics and biology; see
for instance \cite{AronsonWeinberger.1975, McKean.1975, DerridaSpohn.1988,
Murray.2002, vanSaarloos.2003, Munier.2015}. The archetypal model is the Fisher-KPP equation \cite{Fisher.1937,KPP.1937}
\begin{equation}
\partial_t h=\partial_x^2h +h-F(h),
\label{FKPP}
\end{equation}
where, throughout the paper, the non-linearity $F(h)$ is assumed to
satisfy the so-called ``Bramson's conditions'': \cite{Bramson.1978,Bramson.1983}
\begin{equation}
\begin{gathered}
F\in C^1[0,1],\quad F(0)=0,\quad F(1)=1,\quad F'(h)\ge0,\quad F(h)<h\text{ for
$h\in(0,1)$},
\\
F'(h)=\mathcal O(h^p)\text{ for some $p>0$ as $h\searrow0$}.
\end{gathered}
\label{BramsonsConditions}
\end{equation}
(The choice $F(h)=h^2$ is often made.)
One checks with these conditions that $h=0$ is an unstable solution and $h=1$
is a stable solution. We always assume implicitly that the initial
condition $h_0$ satisfies
$h_0\in[0,1]$; this implies, by comparison, that $0<h(x,t)<1$ for all $x$
and all $t>0$. We also always assume for simplicity that $h_0(x)\to1$ as
$x\to-\infty$, but this could be significantly relaxed.

A famous result due to Bramson \cite{Bramson.1978,Bramson.1983} (see also \cite{HamelNolenRoquejoffreRyzhik.2013,Roberts.2013}) states that,
\begin{equation}\label{Bramson's result}
h\big(2t-\tfrac32\log t + z,t\big) \xrightarrow[t\to\infty]{}\omega(z-a)
\qquad\text{iff $\int \diffd x\, h_0(x) x e^x < \infty$},
\end{equation}
where $\omega(z)$,
called the critical travelling wave, is a decreasing function
interpolating from $\omega(-\infty)=1$ to $\omega(+\infty)=0$,
and where the shift $a$ depends on the initial condition. In words, if
$h_0$ decays ``fast enough'' at infinity, then the stable solution $h=1$
on the left invades the unstable solution $h=0$ on the right, and the
position of the invasion front is $2t-\frac32\log t + a$.

The travelling wave $\omega$ is the unique solution to
\begin{equation}
\omega''+2\omega'+\omega-F(\omega)=0,\qquad
\omega(-\infty)=1,\qquad\omega(0)=\tfrac12,\qquad\omega(+\infty)=0,
\label{propomega}
\end{equation}
and there exists $\tilde\alpha>0$ and $\tilde\beta\in\R$
(depending on the choice of the non-linearity $F(h)$) 
such that
$\omega(z)= (\tilde\alpha z + \tilde\beta ) e^{-z}+\mathcal O(e^{-(1+q)z})$ as
$z\to\infty$, where $q$ is any number in $(0,p)$. We prefer to write the equivalent statement:
\begin{equation}
\omega(z-a)= (\alpha z + \beta ) e^{-z}+\mathcal O(e^{-(1+q)z})\quad\text{as
$z\to\infty$},
\label{omegaalphabeta}
\end{equation}
where $\alpha>0$ and $\beta$ now depend also on the initial condition
$h_0$ through $a$
and are given by
$\alpha=\tilde\alpha e^a$ and
$\beta=(\tilde\beta-a\tilde\alpha)e^a$.

Let $\mu_t$ be the position where the front at time $t$ has value $1/2$ (or
the largest such position if there are more than one):
\begin{equation}
h(\mu_t,t)=\frac12.
\label{defmut}
\end{equation}
Bramson's result \eqref{Bramson's result} implies that
$\mu_t=2t-\frac32\log t + a + o(1)$ for large times if $\int \diffd x\,
h_0(x) x e^x<\infty$.
Recent results indicate that a more precise estimate of $\mu_t$ can be
given: if $h_0$ decays to zero ``fast enough'' as
$x\to\infty$, the position $\mu_t$ of the front is believed to satisfy:
\begin{equation}\label{position}
\mu_t=2t-\frac32\log t +a -3\frac{\sqrt\pi}{\sqrt t}+\frac98[5-6\log
2]\frac{\log t}t+\mathcal O\Big(\frac1t\Big),
\end{equation}
where we recall that $a$ depends on the initial condition and on the choice
of $F(h)$.
The $1/\sqrt
t$ correction is known as the Ebert-van Saarloos correction, from
a non-rigorous physics paper \cite{EbertvanSaarloos.2000}. This result was
proved \cite{NolenRoquejoffreRyzhik.2019} for $F(h)=h^2$ and $h_0$ a compact
perturbation of the step function (\textit{i.e.\@} $h_0$ differs from
the step function $\indic{x<0}$ on a compact set); see also
\cite{BBHR.2016}. The $(\log t)/t$ correction was conjectured in
\cite{BerestyckiBrunetDerrida.2017,BerestyckiBrunetDerrida.2018} using
universality argument and a implicit solution of a related model; it was
proved  in \cite{Graham.2019} for $F(h)=h^2$ and  $h_0$ a compact 
perturbation of the step function. Arguments given in \cite{BerestyckiBrunetDerrida.2017} suggest that
the Ebert-van Saarloos term holds iff  $\int \diffd x\, h_0(x) x^2 e^x
<\infty$ and that the $(\log t)/t$ terms holds iff $\int \diffd x\, h_0(x) x^3 e^x
<\infty$, for any choice of $F(h)$ satisfying~\eqref{BramsonsConditions}.

Another quantity of interest is the value of $h(ct,t)$ for $c>2$ and
large $t$. For instance, recalling \cite{McKean.1975} that $h(x,t)$, for
$F(h)=h^2$ and $h_0=\indic{x<0}$, is the probability that the rightmost
position at time $t$  in a branching Brownian motion is located on the
right of $x$, then $h(ct,t)$ would be the probability of a large deviation
where this rightmost position sustains a velocity $c>2$ some time $t$.

For a step initial condition, it is known
\cite{ChauvinRouault.1988, BovierHartung.2014, BovierHartung.2015, DerridaMeersonSasorov.2016,
BerestyckiBrunetCortinesMallein.2022} that
\begin{equation}\label{Phi}
h(ct,t) \sim \Phi(c) \frac1{\sqrt{4\pi t}} e^{(1-\frac{c^2}4)t}
\quad\text{as $t\to\infty$, for $c>2$},
\end{equation}
for some continuous function $c\mapsto \Phi(c)$, for an arbitrary
non-linearity $F(h)$ \cite{ChauvinRouault.1988,}. 
(Note: The function $\Phi(c)$ in \eqref{Phi} is defined as in
\cite{DerridaMeersonSasorov.2016}. The function $\tilde C(\sigma_e)$ in \cite{BovierHartung.2015} and
$C(\rho)$ in \cite{BerestyckiBrunetCortinesMallein.2022} are identical and
related to $\Phi(c)$ by $\frac1{\sqrt{4\pi}}\Phi(c)=\frac2c C(\frac c2)$.)

We show in \autoref{prop1} below that \eqref{Phi} actually holds for any initial condition $h_0$ and any
$c>2$ such that $\int\diffd x \,h_0(x) e^{\frac c2 x}<\infty$, with
a function $\Phi(c)$ depending of course on $h_0$ and on $F(h)$.

The time dependence in
\eqref{Phi} is not surprising: the solution $h_\text{lin}$ to the
linearised Fisher-KPP equation, \textit{i.e.}\@
\eqref{FKPP} with $F(h)=0$, and with a step initial condition
$h_0(x)=\indic{x<0}$ satisfies \eqref{Phi} with a prefactor
$\Phi_\text{lin}(c)=2/c$. However, the dependence in $c$ of the prefactor
$\Phi(c)$ for the (non-linear) Fisher-KPP equation is much more
complicated. For $h_0(x)=\indic{x<0}$, it has been proved
\cite{BovierHartung.2015,BerestyckiBrunetCortinesMallein.2022} that
\begin{equation}
\Phi(2)=0,\qquad \Phi(c)\sim\frac2c \quad\text{as $c\to\infty$}.
\end{equation}
It is argued in \cite{DerridaMeersonSasorov.2016} that, for
$h_0(x)=\indic{x<0}$ and $F(h)=h^2$,
\begin{equation}
\Phi(2+\epsilon)\sim2\sqrt\pi \alpha \epsilon\quad\text{as
$\epsilon\searrow0$},\qquad\qquad
\Phi(c)\simeq\frac2c-\frac8{c^3}+\frac{6.818\ldots}{c^5}+\cdots
\quad\text{as $c\to\infty$},
\end{equation}
where $\alpha$
is the coefficient defined in \eqref{omegaalphabeta}.

The main result of this paper is an asymptotic expansion of the function
$\Phi$ for $c$ close to $2$:

\begin{thm}\label{mainthm}
For the Fisher-KPP equation \eqref{FKPP} with $F(h)=h^2$ and an initial
condition $h_0$ which is a compact perturbation of the step function, one
has
\begin{equation}
\begin{aligned}
\Phi(2+\epsilon)
=\sqrt\pi\Big(\alpha-\frac\beta2\epsilon\Big)
\Big[ 2\epsilon+3\epsilon^2\log\epsilon
-3\Big(1-\frac{\gamma_E}2\Big)\epsilon^2
+\frac94\epsilon^3\log^2\epsilon
& \\
+\frac34(3\gamma_E-6\log2-1)\epsilon^3\log\epsilon
&\Big]+\mO(\epsilon^3)
\end{aligned}
\label{mainresult}
\end{equation}
where $\gamma_E$ is Euler's constant, and where $\alpha$ and $\beta$
 are the coefficients defined in \eqref{omegaalphabeta}.

Actually, \eqref{mainresult} holds for any choice of $F(h)$ and of $h_0$
such that
\begin{enumerate}
\item $\int \diffd x\, h_0(x) e^{rx}<\infty$ for some $r>1$. (Otherwise,
$\Phi(c)$
would not be defined for $c>2$ and the expansion \eqref{mainresult} would be
meaningless.)
\item The position $\mu_t$ of the front satisfies the expansion~\eqref{position},
\item There exists $C>0$, $t_0\ge0$ and a neighbourhood $U$ of $1$ such that
\begin{equation}
\left |
\int \diffd z \, \Big(F[h(\mu_t+z,t)]
- F[\omega(z)]\Big)
 e^{rz}
\right| \le \frac C t\qquad\text{for $t> t_0$ and $r\in U$}.
\label{technical}
\end{equation}
\end{enumerate}
\end{thm}
As will be apparent in the proofs, the expansion \eqref{mainresult} for
$\Phi(2+\epsilon)$ is closely related to the expansion~\eqref{position} for the
position $\mu_t$; in some sense, the 
$\epsilon^2\log\epsilon$ and $\epsilon^3\log\epsilon^2$ terms in
\eqref{mainresult} are connected
to the $1 /\sqrt t$ term in \eqref{position}, and the
$\epsilon^3\log\epsilon$ to the $(\log t)/t$ term.

We will also see in the proof that \eqref{position} cannot hold unless $\int
\diffd x\, h_0(x) x^3e^x<\infty$. As already mentioned, we expect  the converse to be true. 

The technical condition \eqref{technical} should not be surprising:
The quantity $\delta(z,t):=h(\mu_t+z,t)-\omega(z)$ goes to zero
as $t\to\infty$. Moreover, it
satisfies
$\partial_t \delta = \partial_x^2\delta +\dot\mu_t \partial_x\delta +\delta
-F(\omega+\delta)+F(\omega) +(\dot\mu_t-2)\omega'$. For large times, one
can expect from \eqref{position} that $\dot\mu_t-2\sim -\frac3{2t}$ and
$\partial_t\delta\simeq
\partial_x^2\delta+2\partial_x\delta+\delta-F'(\omega)\delta
-\frac3{2t}\omega'$.
Then, it seems likely
that $\delta(z,t)\sim\frac1t\psi(z)$ with
$\psi$ a solution to $\psi''+2\psi'+\psi-F'(\omega)\psi=\frac32\omega'$.
(This is actually a result of \cite{Graham.2019} in the case $F(h)=h^2$.)
This leads to, $F[h(\mu_t+z,t)]-F[\omega(z)]\sim \delta(z,t) F'[\omega(z)]
\sim\frac1t\psi(z) F'[\omega(z)]$, of order $1/t$. Furthermore, (ignoring
polynomial prefactors), $\psi(z)$ decreases as $e^{-z}$ for large $z$ and
$F'[\omega(z)]$ should roughly decrease as $e^{-pz}$, see \eqref{BramsonsConditions}, so
that the integral in \eqref{technical} should converge quickly for $r$
around 1 for $z\to\pm\infty$, and give a result of order $1/t$.

In terms of the function $C(\rho)$ defined in
\cite{BerestyckiBrunetCortinesMallein.2022}, our result can be written as
\begin{equation}
C(1+\epsilon)
=(\alpha-\beta\epsilon)\Big[2\epsilon+ 6\epsilon^2\log\epsilon
+(3\gamma_E+6\log2-4)
\epsilon^2
+9\epsilon^3\log^2\epsilon
+3(3\gamma_E+1)\epsilon^3\log\epsilon
\Big]+\mO(\epsilon^3).
\end{equation}
Note that the authors write $C(\rho)\sim\alpha(\rho-1)$ as $\rho\searrow1$
(bottom of p.\,2095),
but their $\alpha$ is twice ours.

The first part of \autoref{mainthm} is the consequence of its second part
and of the following result:

\begin{prop}[Mostly Cole Graham 2019 \cite{Graham.2019}]\label{prop1t}
For the Fisher-KPP equation \eqref{FKPP} with $F(h)=h^2$ and an initial
condition $h_0$ which is a compact perturbation of the step function,
\eqref{position} and \eqref{technical} hold.
\end{prop}
The fact that \eqref{position} holds under the hypotheses of
\autoref{prop1t} is the main result of \cite{Graham.2019}. The proofs
of \cite{Graham.2019} contain the hard parts 
in showing that \eqref{technical} also holds.

The main tool used in this paper is the so-called magical relation,
which gives a relation between the initial condition, the position $\mu_t$
of the front, and the non-linear part of the equation. Introduce
\begin{equation}
\label{defgamma}
\gamma:=\sup\Big\{r>0; \int \diffd x\, h_0(x) e^{rx}<\infty\Big\},
\end{equation}
and 
\begin{equation}
\varphi(\epsilon,t)   :=\int\diffd z\, F[h(\mu_t+z,t)]e^{(1+\epsilon)z},\qquad
\hat\varphi(\epsilon) :=\int\diffd z\, F[\omega(z)]   e^{(1+\epsilon)z}.
\label{defphi}
\end{equation}
(With these quantities, the condition \eqref{technical} can  be written
$|\varphi(\epsilon,t)-\hat\varphi(\epsilon)| \le C/t$ for all $t>t_0$ and
all $\epsilon$ in some neighbourhood of 0.)
Then
\begin{prop}[Magical relation]\label{prop2}
For any $\epsilon\in(-1,\gamma-1)$
the following relation holds
\begin{equation}
\int_0^\infty\diffd t \,\varphi(\epsilon,t)e^{-\epsilon^2
t +(1+\epsilon)(\mu_t-2t)}=
\int\diffd x \,h_0(x)e^{(1+\epsilon) x}
-\indic{\epsilon>0}\Phi(2+2\epsilon).
\label{prop21}
\end{equation}
Furthermore, if $\gamma>1$ and \eqref{technical} holds,
one has
\begin{equation}
\hat\varphi(\epsilon)\int_0^\infty\diffd t \,e^{-\epsilon^2
t +(1+\epsilon)(\mu_t-2t)}=
\int\diffd x \,h_0(x)e^{(1+\epsilon) x}
-\indic{\epsilon>0}\Phi(2+2\epsilon)+\mP(\epsilon)+\mO(\epsilon^3).
\label{prop22}
\end{equation}
where $\mP(\epsilon)$ is some polynomial in $\epsilon$.
\end{prop}

The second form \eqref{prop22} gives a relation between
$\mu_t$ and $h_0$ which does not involve the front $h(x,t)$ at any finite
time.  Notice also that the non-linear term $F(h)$ only appears in
$\hat\varphi(\epsilon)$.

The magical relation was introduced in
\cite{BrunetDerrida.2015,BerestyckiBrunetDerrida.2017,BerestyckiBrunetDerrida.2018},
but only for $\epsilon<0$.
It allowed (non-rigorously) to compute the asymptotic expansion of the
position of the front for an arbitrary initial condition, and in
particular to obtain
\eqref{position}. The basic idea is the following:
for $\epsilon<0$, the whole right hand side of \eqref{prop22} can be
written as $\mP(\epsilon)+\mathcal O(\epsilon^3)$ for some polynomial
$\mP(\epsilon)$ if $h_0$ goes to zero fast enough. (Specifically, it can be shown that
the necessary and sufficient condition is $\int\diffd x\,h_0(x) x^3 e^x<\infty$.)
 However, the left
hand side produces very easily some singular terms of $\epsilon$ in a small
$\epsilon$ expansion; it
turns out that $\mu_t$ should satisfy \eqref{position} in order to avoid
all the singular terms up to order $\epsilon^3$.

In this paper, by considering both sides $\epsilon<0$ and
$\epsilon>0$, we can eliminate the unknown polynomial $\mP(\epsilon)$ 
in \eqref{prop22} and obtain \eqref{mainresult}.

The rest of the paper is organized as follow; in \autoref{sec:Phi}, we show that the function $\Phi(c)$ is well
defined, and we give a useful representation. In \autoref{sec:magic}, we
prove the first part of \autoref{prop2}, \textit{i.e.}\@ \eqref{prop21}. We
state and prove some technical lemmas in  \autoref{sec:lem}, which allow us
to finish the proof of  \autoref{prop2} and to prove \autoref{mainthm} in
\autoref{sec:expansion}.
In \autoref{proofprop1t},  we prove \autoref{prop1t}.
Finally, a technical lemma is proved in \autoref{appendix}.

\section{The function \texorpdfstring{$\Phi(c)$}{Φ(c)}}\label{sec:Phi}

\begin{prop}\label{prop1}
For a given initial condition $h_0$ such that $\int \diffd x\, h_0(x)
e^x<\infty$,
let $h(x,t)$ be the solution to \eqref{FKPP}.
For $c\ge2$, the following (finite or infinite) limits exist  and are equal:
\begin{equation}\label{prop11}
\Phi(c):=\lim_{t\to\infty} \sqrt{4\pi t}\,h(ct,t)e^{\big(\frac{c^2}4-1\big)t}
= \lim_{t\to\infty} e^{-t\big(1+\frac {c^2} 4\big)}\int\diffd x\, h(x,t)
e^{\frac c 2 x}\in[0,\infty].
\end{equation}
Furthermore, $\Phi(2)=0$, $\Phi(c)>0$ for $c>2$ and
\begin{equation}
\Phi(c)
<\infty
\quad\iff\quad
\int\diffd x\, h_0(x)e^{\frac c2 x}<\infty.
\label{Phih0}
\end{equation}
The function $c\mapsto \Phi(c)$ is continuous in the domain where it is
finite.
\end{prop}
\noindent\textbf{Remark:} the condition 
$\int \diffd x\,h_0(x)e^{ x}<\infty$ implies, in
particular,  that the front has a velocity 2.

Before doing a rigorous proof, here is a quick and dirty argument to show
that the second limit in \eqref{prop11} is equal to the first:
starting from the integral in that limit, make the change of variable
$x=vt$ (with $v$ being the new variable) and boldly replace $h(vt,t)$ under the integral sign using
the equivalent implied by the first limit to obtain
\begin{equation}\begin{aligned}
\int \diffd x\,h(x,t)e^{\frac  c2x}&=t\int\diffd
v\, h(vt,t)e^{\frac  c2vt}\\
&\simeq \frac t{\sqrt{4\pi t}}\int\diffd
v\,\Phi(v) e^{(1-\frac{v^2}4+\frac c2v)t}=
e^{(1+\frac{c^2}4)t}\frac{\sqrt t}{\sqrt{4\pi}}\int\diffd
v\,\Phi(v) e^{-\frac14(v-c)^2t}.
\end{aligned}
\end{equation}
(The fact that the substitution only makes sense for $v\ge2$ is not a problem
since, clearly, the part of the integral for $v<2$ does not contribute
significantly.)
The remaining integral is dominated by $v$ close to $c$ in the large time
limit. Replacing $\Phi(v)$ by $\Phi(c)$ and computing the remaining
Gaussian integral gives the second limit.

We will need in the proof a bound on how $h(x,t)$ decreases for large $x$:
For $r>0$, introduce
\begin{equation}
g(r,t):=\int \diffd x\, h(x,t)e^{rx}.
\label{defg}
\end{equation}

\begin{lem}
\label{lem1}
For all $x$, all $t>0$, and all $r>0$ such that $g(r,0)=\int \diffd x\,
h_0(x)e^{rx}<\infty$,
\begin{equation}
h(x,t)\le \frac{e^{(1+r^2)t}}{\sqrt{4\pi t}}g(r,0)e^{-rx},\qquad
g(r,t) \le e^{(1+r^2)t} g(r,0).
\label{lem11}
\end{equation}
\end{lem}
\begin{proof}
Using the comparison principle, one obtains that $h(x,t)\le h_\text{lin}(x,t)$, where
$h_\text{lin}(x,t)$ is the solution to $\partial_t
h_\text{lin}=\partial_x^2 h_\text{lin}+h_\text{lin}$ with initial condition
$h_0$. Solving for $h_\text{lin}$, we get
\begin{equation}
h(x,t)\le \frac{e^t}{\sqrt{4\pi t}} \int \diffd y \, h_0(y)
e^{-\frac{(x-y)^2}{4t}},
\label{29}
\end{equation}
and then,
\begin{equation}
h(x,t)e^{rx}\le \frac{e^t}{\sqrt{4\pi t}} \int \diffd y \, h_0(y) e^{ry}\times e^{r(x-y)-\frac{(x-y)^2}{4t}}
=\frac{e^{(1+r^2)t}}{\sqrt{4\pi t}} \int \diffd y \, h_0(y) e^{ry}\times
e^{-\frac{(x-y-2rt)^2}{4t}}.
\end{equation}
Both inequalities in \eqref{lem11} are obtained from that last relation,
respectively by writing that the Gaussian term is smaller than 1, or by
integrating over $x$.
\end{proof}

\begin{proof}[Proof of \autoref{prop1}]
We write the non-linearity in \eqref{FKPP} as $F(h)=h\times
G(h)$. From \eqref{BramsonsConditions}, the function $G$, defined on $[0,1]$, is continuous, satisfies $0\le G(h)\le 1$, $G(h)=\mathcal
O(h^p)$ for some $p>0$ as $h\to0$ and $G(1)=1$. To avoid parentheses, we will write
$G(h(x,t))$ as $G\circ h(x,t)$ using the composition operator $\circ $.
We write the solution $h(x,t)$ of \eqref{FKPP} using the Feynman-Kac
representation (see \cite[Theorem 5.3 p. 148]{Friedman.1975} or, for
a short proof, \cite[Proposition 3.1]{BerestyckiBrunetPenington.2019}):
\begin{equation}\label{FK1}
h(x,t)=e^t \E_x\Big[h_0(B_t) e^{-\int_0^t\diffd
s\,G\circ h(B_s,t-s)}\Big],
\end{equation}
where under $\E_x$,  $B$ is a Brownian with diffusivity $\sqrt2$
started from $x$.
(so that $\E_x(B_t^2)=x^2+2t$.)

In \eqref{FK1}, we condition the Brownian to end at $B_t=y$ and we
integrate over $y$:
\begin{equation}
h(x,t)=e^t\int\diffd y \frac{e^{-\frac{(x-y)^2}{4t}}}
{\sqrt{4\pi t}}h_0(y) \E_{t:x\to y}\Big[e^{-\int_0^t\diffd s\,G\circ h(B_s,t-s)}\Big]
\end{equation}
where, under $\E_{t:x\to y}$, $B$ is a Brownian bridge going from $x$ to $y$
in a time $t$, with a diffusivity $\sqrt2$. We reverse time and remove the
linear part from the bridge:
\begin{align}
h(x,t)&=e^t\int\diffd y \frac{e^{-\frac{(x-y)^2}{4t}}}
{\sqrt{4\pi t}}h_0(y) \E_{t:y\to x}\Big[e^{-\int_0^t\diffd s\,G\circ h(B_s,s)}\Big]
\label{FK2}
\\
&=\int\diffd y \frac{e^{t-\frac{(x-y)^2}{4t}}}{\sqrt{4\pi t}}h_0(y)
        \E_{t:0\to 0}\Big[e^{-\int_0^t\diffd s\,G\circ h\big(B_s+(x-y)\tfrac
st+y,s\big)}\Big].
\end{align}
Then, at $x=ct$,
\begin{equation}
h(ct,t)=\frac{e^{(1-\frac{c^2}4)t}}{\sqrt{4\pi t}}
\int\diffd y\,e^{\frac c 2y-\frac{y^2}{4t}}
h_0(y)\E_{t:0\to 0}\Big[e^{ -\int_0^t\diffd s\,G\circ h\big(B_s+c s-y\tfrac
st+y,s\big)}\Big]
\end{equation}
We move the prefactors to the left hand side and write
the Brownian bridge as
a time-changed Brownian path
\begin{equation}
\sqrt{4\pi t}\, e^{(\frac{c^2}4-1)t}
h(ct,t)=
\int\diffd y\,e^{\frac c 2y-\frac{y^2}{4t}}
h_0(y)\E_{0}\Big[e^{ -\int_0^t\diffd s\,G\circ
h\big(\frac{t-s}tB_{\frac{ts}{t-s}}+c s-y\tfrac
st+y,s\big)}\Big].
\label{hctt}
\end{equation}

For any $y$, any $c\ge2$, and almost all Brownian path $B$, one has 
\begin{equation}
\int_0^t\diffd s\,G\circ
h\Big(\frac{t-s}tB_{\frac{ts}{t-s}}+c s-y\tfrac
st+y,s\Big) \to \int_0^\infty\diffd s\,G\circ h\big(B_{s}+c s+y,s\big)
\qquad\text{as $t\to\infty$}.
\label{domconv}
\end{equation}
Indeed, first consider the case $c>2$, pick $\tilde c \in (2,c)$ and $t_0$
such that $c-y/t_0>\tilde c$.
Recall that, for almost all path $B$, there exists a constant $A$
(depending on $B$) such that $|B_u|\le
A(1+u^{0.51})$ for all $u$; this implies that
$\frac{t-s}t\big|B_{\frac{ts}{t-s}}\big|\le
A(1+s^{0.51})$ for all $t$ and all $s<t$. Then
\begin{equation}
\frac{t-s}tB_{\frac{ts}{t-s}}+c s-y\tfrac
st+y \ge \tilde c s + C+y\quad\text{for all $t>t_0$ and all $s\in(0,t)$,}
\end{equation}
where $C$  is some constant depending on
$B$. (Indeed, the function $s\mapsto cs-y s/t-\tilde
c s -As^{0.51}$ is uniformly bounded from below for $t>t_0$.)
Using 
\eqref{lem11} for $r=1$, we obtain that 
\begin{equation}
h(\frac{t-s}tB_{\frac{ts}{t-s}}+c s-y\tfrac
st+y,s)\le \frac{C}{\sqrt s} e^{-(\tilde c-2)s-y}\quad\text{for all $t>t_0$ and
all $s\in(0,t)$},
\end{equation}
 with $C$ another
constant depending on $B$.
As $G(h)=\mO(h^p)$ for some $p>0$ as $h\to0$ and $G(1)=1$, there exists
a constant $C$ such that $G(h)\le C h^{\min(1,p)}$. Then, we see by
dominated convergence that \eqref{domconv} holds for $c>2$, and furthermore
we see that the right hand side is smaller than $C e^{-\min(1,p)y}$ for
some constant $C$ depending on $B$.

For $c=2$, the right hand side of \eqref{domconv} is $+\infty$. Indeed,
$B_s+cs+y$ is infinitely often smaller than $2s-\sqrt s$, where the front
$h$ is close to 1. Then, noticing that \eqref{domconv} with the upper limits
of both integrals replaced by some $T>0$ clearly holds by dominated
convergence, and that, by choosing $T$ large enough, the right hand side is
arbitrarily large, we see that the left hand side of \eqref{domconv} must
diverge as $t\to\infty$.

From \eqref{domconv}, we immediately obtain by dominated convergence
\begin{equation}
\E_{0}\Big[e^{ -\int_0^t\diffd s\,G\circ
h\big(\frac{t-s}tB_{\frac{ts}{t-s}}+c s-y\tfrac
st+y,s\big) }\Big] \to
\E_{0}\Big[e^{ - \int_0^\infty\diffd s\,G\circ h(B_{s}+c s+y,s)
}\Big] \qquad\text{as $t\to\infty$},
\end{equation}
where the right hand side is 0 if $c=2$ and positive if $c>2$.
(As we have shown, the integral in the exponential is almost surely
infinite if $c=2$, and almost surely finite if $c>2$.) Furthermore, for
$c>2$, the right hand side converges to 1 as $y\to\infty$. (Recall that,
for $c>2$, the integral in the exponential is smaller than
$Ce^{-\min(1,p)y}$.)

If $\int h_0(y)e^{cy/2}\diffd y<\infty$, then a last application of
dominated convergence in \eqref{hctt} shows that the first limit defining
$\Phi(c)$ in \eqref{prop11} does exist and is
given by:
\begin{equation}
\Phi(c):=\lim_{t\to\infty} {\sqrt{4\pi t}}\, e^{(\frac{c^2}4-1)t}h(ct,t)         =\int\diffd y\,e^{\frac c 2y}
h_0(y)\E_0\Big[e^{ -\int_0^\infty\diffd s\,G\circ h(B_s+c s
+y,s)}\Big]<\infty,
\label{1stexpressionphi}
\end{equation}
and furthermore $\Phi(2)=0$ and $\Phi(c)>0$ for $c>2$. Note that \cite{BerestyckiBrunetCortinesMallein.2022} gives a similar
expression.

We now assume that $\int h_0(y)e^{cy/2}\diffd y=\infty$ and show that the
limit of \eqref{hctt}
diverges. Notice that we must be in the $c>2$ case since we also assumed
that $\int h_0(y)e^y\,\diffd y<\infty$.
Cutting the integral in \eqref{hctt} at some arbitrary value $A$ and then 
sending $t\to\infty$ gives
\begin{equation}
\liminf_{t\to\infty} 
 {\sqrt{4\pi t}}\, e^{(\frac{c^2}4-1)t}h(ct, t)
=\int_{-\infty}^A\diffd y\,e^{\frac c 2y}
h_0(y)\E_0\Big[e^{ -\int_0^\infty\diffd s\,G\circ h(B_s+c s
+y,s)}\Big].
\label{33}
\end{equation}
As the expectation appearing in the integral goes to
1 as $y\to\infty$, the hypothesis  $\int h_0(y)e^{cy/2}\diffd y=\infty$
implies that the right hand side diverges as $A$ to $\infty$, and then that
$\Phi(c)$ exists and is infinite.

Using the same methods, one can show from \eqref{1stexpressionphi} that $\Phi(c)$ is a continuous
function (in the range of $c$ where $\Phi$ is finite)
 by first showing that $\int_0^\infty\diffd
s\,G\circ h(B_s+c_n s +y,s)\to\int_0^\infty\diffd s\,G\circ
h(B_s+c s +y,s)$ if $c_n\to c$, by dominated convergence, using the
same bounds as above (specifically that the integrands are uniformly bounded by an
exponentially decreasing function of $s$ if $c>2$ and that the result is
infinity if $c=2$.) 

To show that the second expression in \eqref{prop11} is equal to the first, start again from \eqref{FK2}:
$$\begin{aligned}
h(x,t)=e^t\int\diffd y \frac{e^{-\frac{(x-y)^2}{4t}}}
{\sqrt{4\pi t}}h_0(y) \E_{t:y\to x}\Big[e^{-\int_0^t\diffd s\,G\circ h(B_s,s)}\Big]
=e^t\int\diffd y \,
h_0(y) \E_{y}\Big[e^{-\int_0^t\diffd s\,G\circ h(B_s,s)}\delta(B_t-x)\Big].
\end{aligned}$$

Then
\begin{equation}
\begin{aligned}
e^{-t\big(1+\frac {c^2} 4\big)}\int\diffd x\, h(x,t) e^{\frac c 2 x}
&=e^{-\frac{c^2}4t}\int\diffd y \,
h_0(y) \E_{y}\Big[e^{-\int_0^t\diffd s\,G\circ h(B_s,s)}e^{\frac c 2 B_t}\Big]
\\&=\int\diffd y\,e^{\frac
c2 y}h_0(y)\E_y\Big[\text
e^{-\int_0^t\diffd s\,G\circ h(B_s+cs,s)}\Big].
\end{aligned}\end{equation}
where the last transform is through Girsanov's Theorem (or a change of
probability of the Brownian).
Taking the limit $t\to\infty$ is immediate and gives back the
expression of $\Phi(c)$ written in \eqref{1stexpressionphi}.
\end{proof}

\section{Magical relation}\label{sec:magic}
In this section, we prove the first part of \autoref{prop2}. Many of the
arguments presented here were already written in
\cite{BerestyckiBrunetDerrida.2018}
for the case $\epsilon<0$.

Recall the definitions \eqref{defgamma} of $\gamma$ and \eqref{defg} of
$g(r,t)$:
\begin{equation}
\gamma:=\sup\Big\{r; \int \diffd x\, h_0(x) e^{rx}<\infty\Big\},
\qquad
g(r,t):=\int \diffd x\, h(x,t) e^{rx}.
\end{equation}
According to \autoref{lem1},
\begin{equation}
g(r,t)\le e^{(1+r^2)t}g(r,0)<\infty\qquad\text{for $r\in(0,\gamma)$ and  $t\ge0$}. 
\end{equation}

We wish to write an expression for $\partial_t g(r,t)$, and the first step
is to justify that we can differentiate under the integral sign:
\begin{equation}
\partial_t g(r,t) =\int\diffd x\,\partial_t h(x,t) e^{rx}=\int\diffd x\,\big[\partial_x^2 h(x,t) + h(x,t)-F[h(x,t)]\big] e^{rx}
\qquad\text{for $0<r<\gamma$},
\label{36}
\end{equation}
and then (still assuming $0<r<\gamma$) that we can integrate twice by parts
the $\partial_x^2h$ term:
\begin{equation}
\partial_t g(r,t) =\int\diffd x\, \big[(r^2 h(x,t) + h(x,t)-F[h(x,t)]\big] e^{rx}
=(1+r^2)g(r,t) -\int\diffd x\,F[h(x,t)]e^{rx}.
\label{37}
\end{equation}
Both steps \eqref{36} and \eqref{37} are justified by using bounding
functions provided by the following
lemma with $\beta$ chosen in $(r,\gamma)$:
\begin{lem}\label{lem2}
Let $\beta\in(0,\gamma)$. For $t>0$, the quantities $h(x,t)$, $|\partial_x
h(x,t)|$, $|\partial_x^2h(x,t)|$ and $|\partial_t h(x,t)|$ are bounded by
$A(t)\max(1,e^{-\beta x})$ for some locally bounded function $A$.
\end{lem}
The proof of \autoref{lem2} is given in \autoref{appendix}.

Recall  the definition \eqref{defphi} of $\varphi$:
\begin{equation}
\varphi(\epsilon,t)   :=\int\diffd z\, F[h(\mu_t+z,t)]e^{(1+\epsilon)z};
\end{equation}
we have
\begin{equation}
\int\diffd x\,F[h(x,t)]e^{rx}=e^{r\mu_t}\int\diffd
z\,F[h(\mu_t+z,t)]e^{rz}=e^{r\mu_t}\varphi(r-1,t),
\end{equation}
and so, in \eqref{37},
\begin{equation}
\partial_t g(r,t) =(1+r^2)g(r,t) - e^{r\mu_t}\varphi(r-1,t).
\end{equation}
Integrating, we obtain
\begin{equation}
g(r,t)e^{-(1+r^2)t}
=g(r,0) - \int_0^t \diffd s\, \varphi(r-1,s) e^{r\mu_s-(1+r^2)s}.
\label{54}
\end{equation}
We now send $t\to\infty$ in \eqref{54}, distinguishing two cases
\begin{itemize}
\item If $\gamma>1$ and $1\le r<\gamma$;\quad notice that the left hand
side  is the expression appearing in the second limit in
\eqref{prop11} with $c=2r$. This implies that
\begin{equation}
\Phi(2r) = g(r,0) -\int_0^\infty \diffd s\, \varphi(r-1,s) e^{r\mu_s-(1+r^2)s}
\qquad\text{if $1\le r<\gamma$}.
\label{r>1}
\end{equation}
\item If $0<r<\min(1,\gamma)$;\quad we claim that the left hand
side of \eqref{54} goes to 0 as $t\to\infty$, and so:
\begin{equation}
0 = g(r,0) -\int_0^\infty \diffd s\, \varphi(r-1,s) e^{r\mu_s-(1+r^2)s}
\qquad\text{if $0<r<\min(1,\gamma)$}.
\label{r<1}
\end{equation}
Indeed,
take $\beta\in\big(r,\min(1,\gamma)\big)$.
Applying \eqref{lem11} with $\beta$ instead of $r$, we have
\begin{equation}
h(x,t)\le\min\left[1, \frac{e^{(1+\beta^2)t}}{\sqrt{4\pi
t}}g(\beta,0)e^{-\beta x}\right].
\label{hmin}
\end{equation}
For $t$ given, let $X$ be the point where both expressions inside the min are equal:
\begin{equation}
e^{\beta X}=\frac{e^{(1+\beta^2)t}}{\sqrt{4\pi
t}}g(\beta,0).
\end{equation}
We obtain from \eqref{hmin}
\begin{equation}
g(r,t)=\int\diffd x\,h(x,t)e^{rx}\le
\frac{e^{rX}}r+ \frac{e^{(1+\beta^2)t}}{\sqrt{4\pi
t}}g(\beta,0)\frac{e^{-(\beta-r) X}}{\beta-r}
=\left(\frac1r+\frac1{\beta-r}\right)e^{rX}
=C\frac{e^{r(\beta^{-1}+\beta)t}}{t^{\frac r{2\beta}}}
,
\end{equation}
where $C$ is some quantity depending on $r$ and $\beta$, but independent of
time. As $\beta>r$ and as
the function $\beta\to \beta^{-1}+\beta$ is decreasing for $\beta<1$, we
obtain
$r(\beta^{-1}+\beta)<r(r^{-1}+r)=1+r^2$. We conclude that, indeed, the left hand side of
\eqref{54} goes to zero as $r\to\infty$. 
\end{itemize}

Combining \eqref{r>1} and \eqref{r<1}, we have shown that
\begin{equation}
\int_0^\infty \diffd s\, \varphi(r-1,s) e^{r\mu_s-(1+r^2)s}
 = g(r,0) -\indic{r>1}\Phi(2r)\qquad\text{for $r\in(0,\gamma)$}.
\label{bothr}
\end{equation}
(Recall that $\Phi(2)=0$, hence the right hand side is continuous at
$r=1$.)
Writing now that $r\mu_s-(1+r^2)s= r(\mu_s-2s)-(1-r)^2s$, and taking
$r=1+\epsilon$ and $s=t$ in \eqref{bothr}, we obtain
\begin{equation}
\int_0^\infty\diffd t \,\varphi(\epsilon,t)e^{-\epsilon^2
t +(1+\epsilon)(\mu_t-2t)}=
g(1+\epsilon,0)
-\indic{\epsilon>0}\Phi(2+2\epsilon)\quad\text{for
$\epsilon\in(-1,\gamma-1)$},
\end{equation}
which is the first part \eqref{prop21} of \autoref{prop2}.

To prove the second part of the proposition, we need to show that, for some
polynomial $\mP(\epsilon)$,
\begin{equation}
\int_0^\infty\diffd t \,\big[\varphi(\epsilon,t)-\hat\varphi(\epsilon)\big]e^{-\epsilon^2
t +(1+\epsilon)(\mu_t-2t)}=\mP(\epsilon)+\mO(\epsilon^3),
\label{42}
\end{equation}
if $ \big|\varphi(\epsilon,t)-\hat\varphi(\epsilon)\big|\le\frac C t $ for
$t$ large enough and
$\epsilon$ in some real neighbourhood of 0 and if $\gamma>1$,
which implies that 
$\mu_t = 2t-\frac32\log t +a +o(1)$ (by Bramson's result).
To do so, we need several technical lemmas. We
thus take a pause in the proof of \autoref{prop2} to state and prove these
lemmas, and we resume in \autoref{sec:expansion}.

\section{Some technical lemmas}\label{sec:lem}

We begin by recalling a classical result on the analyticity of functions
defined by an integral:

\begin{lem}\label{lem:a}
Let $f(\epsilon,t)$ be a family of functions such that
\begin{itemize}
\item $\epsilon\mapsto f(\epsilon,t)$ is analytic on some simply
connected open domain $U$ of $\C$ (independent of $t$) for almost all $t\in \R$,
\item $|f(\epsilon,t)|\le g(t)$ for all $\epsilon\in U$, where $g$ is
some integrable function: $\int g(t)\,\diffd t <\infty$.
\end{itemize}
Then, $\epsilon\mapsto F(\epsilon):=\int\diffd t\, f(\epsilon,t)$ is
analytic on $U$.
\end{lem}
\begin{proof}
On any closed path $\gamma$ in $U$,
one has with Fubini
\begin{equation}
\oint_\gamma F(\epsilon) \,\diffd\epsilon=\int\diffd t \oint_\gamma
\diffd\epsilon\,f(\epsilon,t).
\end{equation}
This last integral is 0 since $\epsilon\mapsto f(\epsilon,t)$ is analytic
and $U$ is simply connected. Then, by Morera's theorem, $F$ is analytic.
\end{proof}

The next lemma states that some functions of $\epsilon$ which are variations
on the incomplete gamma functions have small $\epsilon$ expansions with
only one or two singular terms
\begin{lem}\label{lem:1s}
Let $\alpha,\beta$ be real numbers such that either
$\alpha\not\in\{1,2,3,\ldots\}$, or $\beta\ne0$. There exist
functions $\epsilon\mapsto \mA_{\alpha,\beta}(\epsilon)$ and
$\epsilon\mapsto \tilde \mA_{\alpha,\beta}(\epsilon)$ which are analytic around
$\epsilon=0$, such that for $\epsilon$ real, non-zero and $|\epsilon|$
small enough,
\begin{align}\label{lemeq1}
\int_1^\infty \diffd t\, e^{-\epsilon^2 t}
\frac1{t^{\alpha+\beta\epsilon}}&=
\lvert\epsilon\rvert^{2\alpha-2+2\beta\epsilon}\Gamma(1-\alpha-\beta\epsilon)
+\frac{\indic{\alpha=1}}{\beta\epsilon}
+\mA_{\alpha,\beta}(\epsilon),
\\
\label{lemeq2}
\int_1^\infty \diffd t\, e^{-\epsilon^2 t}
\frac{\log t}{t^{\alpha+\beta\epsilon}}&=
\lvert\epsilon\rvert^{2\alpha-2+2\beta\epsilon}\Big[-2\log\lvert\epsilon\rvert
\Gamma(1-\alpha-\beta\epsilon)+\Gamma'(1-\alpha-\beta\epsilon)\Big]
+\frac{\indic{\alpha=1}}{(\beta\epsilon)^2}+\tilde \mA_{\alpha,\beta}(\epsilon).
\end{align}
\end{lem}
\noindent\textbf{Remarks:} the condition $\alpha\not\in\{1,2,3,\ldots\}$ or
$\beta\ne0$ ensures that the gamma functions appearing in the result are
well defined for $\epsilon$ small enough.  For $\alpha=n\in\{1,2,3,\ldots\}$ and $\beta=0$
one would have
 $$\int_1^\infty \diffd t\, e^{-\epsilon^2 t} \frac1{t^{n}}
=\frac{2(-1)^n}{(n-1)!} \epsilon^{2n-2}\log|\epsilon|+\mathcal A_{n,0}(\epsilon),$$
but we don't need this result in the present paper, and we skip the proof.
 For convenience, we give the
results we actually use, writing simply $\mA(\epsilon)$ for
the analytic functions:
\begin{align}
\int_1^\infty \diffd t\, e^{-\epsilon^2 t} \frac1{t^{\frac32+\frac32\epsilon}}
&=|\epsilon|^{1+3\epsilon}\Gamma(-\tfrac12-\tfrac32\epsilon)
+\mA(\epsilon),
\qquad\int_1^\infty \diffd t\, e^{-\epsilon^2 t} \frac1{t^{2+\frac32\epsilon}}
=|\epsilon|^{2+3\epsilon}\Gamma(-1-\tfrac32\epsilon) +\mA(\epsilon),
\notag\\
\int_1^\infty \diffd t\, e^{-\epsilon^2 t} \frac{\log t}{t^{\frac52+\frac32\epsilon}}
&=|\epsilon|^{3+3\epsilon}\Big[-2\log|\epsilon|\Gamma(-\tfrac32-\tfrac32\epsilon)+\Gamma'(-\tfrac32-\tfrac32\epsilon)\Big]
+\mA(\epsilon).
\label{lem1seq}
\end{align}

\begin{proof}[Proof of \autoref{lem:1s}]
Fix $\alpha$ and $\beta$ such that either $\alpha\not\in\{1,2,3,\ldots\}$
or $\beta\ne0$.  We restrict $\epsilon$ to be real, non-zero and
$|\epsilon|$ to be small enough so that
$\alpha+\beta\epsilon\not\in\{1,2,3,\ldots\}$.
This ensures that the $\Gamma$ function and its derivative in
\eqref{lemeq1} and \eqref{lemeq2} are defined,
and we define $A_{\alpha,\beta}(\epsilon)$ and $\tilde
A_{\alpha,\beta}(\epsilon)$ by (respectively) \eqref{lemeq1} and
\eqref{lemeq2}. We now show
that the functions thus defined can be extended into 
analytic functions around $\epsilon=0$ for $\alpha<1$.

We first consider $\alpha<1$. Note that by our restriction on the range of
allowed $\epsilon$, one also has $\alpha+\beta\epsilon<1$, and one can
write
\begin{equation}
\int_1^\infty \diffd t\, 
\frac{ e^{-\epsilon^2 t}}{t^{\alpha+\beta\epsilon}}=
\int_0^\infty \diffd t\,
\frac{ e^{-\epsilon^2 t}}{t^{\alpha+\beta\epsilon}}-
\int_0^1\diffd t\, 
\frac{ e^{-\epsilon^2 t}}{t^{\alpha+\beta\epsilon}}
=\lvert\epsilon\rvert^{2\alpha-2+2\beta\epsilon}\Gamma(1-\alpha-\beta\epsilon)-
\int_0^1\diffd t\, 
\frac{ e^{-\epsilon^2 t}}{t^{\alpha+\beta\epsilon}}
.
\end{equation}
By identification with \eqref{lemeq1}, one obtains
\begin{equation}
A_{\alpha,\beta}(\epsilon)=-\int_0^1\diffd t\, e^{-\epsilon^2 t}
\frac1{t^{\alpha+\beta\epsilon}}\qquad\text{for $\alpha<1$}.
\end{equation}
Similarly,
\begin{equation}
\tilde A_{\alpha,\beta}(\epsilon)=-\int_0^1\diffd t\, e^{-\epsilon^2 t}
\frac{\log t}{t^{\alpha+\beta\epsilon}}\qquad\text{for $\alpha<1$}.
\end{equation}
Let $\tilde\alpha\in(\alpha,1)$, and let $U$ a simply connected
neighbourhood of $0$ in $\C$ such that
$\alpha+\beta\Re(\epsilon)<\tilde{\alpha}$ and $|e^{-\epsilon^2t}|<2$ for
all $\epsilon\in U$ and $t\in[0,1]$. One
can apply Lemma~\ref{lem:a} with the bounding function
$g(t)=2(1+|\log t|)/t^{\tilde\alpha}\indic{t\in(0,1)}$ to show that 
$A_{\alpha,\beta}(\epsilon)$ and $\tilde
A_{\alpha,\beta}(\epsilon)$ are analytic around 0.

To extend the result to $\alpha\ge1$, we integrate by parts the left hand
side of \eqref{lemeq1}
\begin{equation}
\int_1^\infty \diffd t\, e^{-\epsilon^2 t}
\frac1{t^{\alpha+\beta\epsilon}}=
\frac1{1-\alpha-\beta\epsilon}\Big[-e^{-\epsilon^2}+\epsilon^2\int_1^\infty\diffd
t\,e^{-\epsilon^2 t}\frac1{t^{\alpha+\beta\epsilon-1}}\Big].
\end{equation}
Then, rewriting the integrals in terms of the functions $\mA_{\alpha,\beta}$
as in \eqref{lemeq1},
\begin{multline}
\lvert\epsilon\rvert^{2\alpha-2+2\beta\epsilon}\Gamma(1-\alpha-\beta\epsilon)
+\frac{\indic{\alpha=1}}{\beta\epsilon}
+\mA_{\alpha,\beta}(\epsilon)
\\
=\frac1{1-\alpha-\beta\epsilon}\Big[-e^{-\epsilon^2}+
\lvert\epsilon\rvert^{2\alpha-2+2\beta\epsilon}\Gamma(2-\alpha-\beta\epsilon)
+\frac{\indic{\alpha=2}}{\beta}\epsilon
+\epsilon^2\mA_{\alpha-1,\beta}(\epsilon)\Big].
\end{multline}
With the property $x\Gamma(x)=\Gamma(x+1)$, the terms with the $\Gamma$
functions cancel and one is left with
\begin{equation}
\mA_{\alpha,\beta}(\epsilon)
=
-\frac{\indic{\alpha=1}}{\beta\epsilon}
+
\frac1{1-\alpha-\beta\epsilon}\Big[-e^{-\epsilon^2}+
\frac{\indic{\alpha=2}}{\beta}\epsilon
+\epsilon^2\mA_{\alpha-1,\beta}(\epsilon)\Big].
\end{equation}
For convenience let us also write the special case $\alpha=1$:
\begin{equation}
\mA_{1,\beta}(\epsilon)
=\frac{e^{-\epsilon^2}-1}{\beta\epsilon}
-\frac\epsilon\beta \mA_{0,\beta}(\epsilon).
\label{lema1}
\end{equation}
It is then clear from these equations that, except for
$(\alpha=1,\beta=0)$ or $(\alpha=2,\beta=0)$,
 one has
\begin{equation}\label{emimplies}
\{ \epsilon\mapsto \mA_{\alpha-1,\beta}(\epsilon)\text{ analytic around $0$}\} \implies
\{ \epsilon\mapsto \mA_{\alpha,\beta}(\epsilon)\text{ analytic around $0$}\}.
\end{equation}
As $\mA_{\alpha,\beta}$ is analytic around 0 for $\alpha<1$,
this implies by induction that $\mA_{\alpha,\beta}$ is analytic around
0 for all $\alpha$ if $\beta\ne0$, and for all
$\alpha\not\in\{1,2,3,\ldots\}$ if $\beta=0$.

We proceed in the same way for $\tilde \mA_{\alpha,\beta}$. Integrating by
parts the integral in \eqref{lemeq2},
\begin{equation}
\int_1^\infty \diffd t\, e^{-\epsilon^2 t}
\frac{\log t}{t^{\alpha+\beta\epsilon}}=
\frac1{1-\alpha-\beta\epsilon}\Big[\epsilon^2\int_1^\infty\diffd
t\,e^{-\epsilon^2 t}\frac{\log t}{t^{\alpha+\beta\epsilon-1}}-
\int_1^\infty \diffd t\, e^{-\epsilon^2 t}
\frac1{t^{\alpha+\beta\epsilon}}\Big].
\end{equation}
We replace all the integrals using \eqref{lemeq1} and \eqref{lemeq2} and
notice, using $x\Gamma(x)=\Gamma(x+1)$ and
$\Gamma(x)+x\Gamma'(x)=\Gamma'(x+1)$, that all the terms involving $\Gamma$
functions cancel, \textit{i.e.}\@:
\begin{multline}
\Big[-2\log\lvert\epsilon\rvert
\Gamma(1-\alpha-\beta\epsilon)+\Gamma'(1-\alpha-\beta\epsilon)\Big]
\\=
\frac1{1-\alpha-\beta\epsilon}\Big[-2\log\lvert\epsilon\rvert
\Gamma(2-\alpha-\beta\epsilon)+\Gamma'(2-\alpha-\beta\epsilon)
-\Gamma(1-\alpha-\beta\epsilon)\Big].
\end{multline}
Then, the remaining terms are
\begin{equation}
\frac{\indic{\alpha=1}}{(\beta\epsilon)^2}+\tilde\mA_{\alpha,\beta}(\epsilon)=\frac1{1-\alpha-\beta\epsilon}\Big[\frac{\indic{\alpha=2}}{\beta^2}+\epsilon^2\tilde
\mA_{\alpha-1,\beta}(\epsilon)-\frac{\indic{\alpha=1}}{\beta\epsilon}-\mA_{\alpha,\beta}(\epsilon)\Big]
\end{equation}
In particular, for $\alpha=1$,
\begin{equation}
\tilde
\mA_{1,\beta}(\epsilon)=-\frac\epsilon{\beta}\tilde
\mA_{0,\beta}(\epsilon)+\frac{\mA_{1,\beta}(\epsilon)}{\beta\epsilon}.
\end{equation}
Notice from \eqref{lema1} that $\mA_{1,\beta}(0)=0$. Hence we have again,
except if $\alpha\in\{1,2,3,\ldots\}$ and $\beta=0$ 
\begin{equation}
\{ \epsilon\mapsto \tilde \mA_{\alpha-1,\beta}(\epsilon)\text{ analytic around $0$}\} \implies
\{ \epsilon\mapsto \tilde \mA_{\alpha,\beta}(\epsilon)\text{ analytic
around $0$}\},
\end{equation}
and the proof is finished in the same way as for $\mA_{\alpha,\beta}$.
\end{proof}

\autoref{lem:1s} gives asymptotic expansions of $e^{-\epsilon^2 t}$ times
exact power laws of $t$. The next lemma deals with the case of approximate
power laws.

\begin{lem}\label{lem:O}
Let $f(\epsilon,t)$ be a family of functions such that, for
a certain neighbourhood $U$ of 0 in $\C$,
\begin{itemize}
\item $\epsilon\mapsto f(\epsilon,t)$ is analytic in $U$ for all $t>0$,
\item $\displaystyle|f(\epsilon,t)| \le
C\min\left(1,\frac1{t^{\alpha+\beta\epsilon}}\right)$ for all $t>0$ and
$\epsilon\in U\cap\R$, where $C>0$, $\alpha$ and $\beta$ are some real constants.
\end{itemize}
Then there exists a polynomial $\mP$ such that, for $\epsilon$ real,
non-zero, and $|\epsilon|$ small enough, 
\begin{equation}
\int_0^\infty\diffd t\, e^{-\epsilon^2t}f(\epsilon,t)
= \mP(\epsilon)+\begin{cases}
\mathcal O\big(|\epsilon|^{2\alpha-2}\big)&\text{if
$\alpha\not\in\{1,2,3,\ldots\}$},\\[1ex]
\mathcal O\big(|\epsilon|^{2\alpha-2}\log|\epsilon|\big)&\text{if
$\alpha\in\{1,2,3,\ldots\}$},
\end{cases}
\qquad\text{as $\epsilon\to0$}.
\end{equation}
\end{lem}
\begin{proof}
For the case $\alpha\le 1$, the polynomial $\mP(\epsilon)$ plays no role as
it is asymptotically smaller than the $\mO$ term. Thus, we simply need to
bound the integral for $\epsilon\in U\cap\R$:
\begin{equation}
\bigg|\int_0^\infty\diffd t\, e^{-\epsilon^2t}f(\epsilon,t)
\bigg|\le C +C\int_1^\infty\diffd t\,
e^{-\epsilon^2t}\frac1{t^{\alpha+\beta\epsilon}}.
\end{equation}
The remaining integral is given by \eqref{lemeq1} except for the case
$\alpha=1$ and $\beta=0$:
\begin{equation}
\int_1^\infty\diffd t\,
e^{-\epsilon^2t}\frac1{t^{\alpha+\beta\epsilon}}=
\begin{cases}
\lvert\epsilon\rvert^{2\alpha-2+2\beta\epsilon}\Gamma(1-\alpha-\beta\epsilon)
+\mA_{\alpha,\beta}(\epsilon)
=\mathcal O(|\epsilon|^{2\alpha-2})&\text{if $\alpha<1$},\\[1ex]
\lvert\epsilon\rvert^{2\beta\epsilon}\Gamma(-\beta\epsilon)
+\frac{1}{\beta\epsilon}
+\mA_{\alpha,\beta}(\epsilon)
=\mathcal O(\log|\epsilon|)&\text{if $\alpha=1$},
\end{cases}
\end{equation}
where we used $|\epsilon|^{2\beta\epsilon}=1+\mO(\epsilon\log|\epsilon|)$ and
$\Gamma(-\beta\epsilon)=\frac{\Gamma(1-\beta\epsilon)}{-\beta\epsilon}=-\frac{1}{\beta\epsilon}+\mO(1)$.
One checks independently that the case $\alpha=1$ and $\beta=0$ gives also
$\mathcal O(\log|\epsilon|)$.

For $\alpha>1$, we proceed by induction. Pick $\tilde\alpha\in(1,\alpha)$,
and make the neighbourhood $U$ of 0 small enough that
$\alpha+\beta\epsilon>\tilde\alpha$ for $\epsilon\in U\cap \R$. Then
\begin{equation}
|f(\epsilon,t)|\le
C\min\left(1,\frac1{t^{\alpha+\beta\epsilon}}\right)
\le
C\min\left(1,\frac1{t^{\tilde\alpha}}\right)\qquad\text{for all $t>0$ and
$\epsilon\in U\cap \R$}.
\end{equation}
Integrating by parts,
\begin{equation}
\int_0^\infty\diffd t\, e^{-\epsilon^2t}f(\epsilon,t)
=F(\epsilon,0)-\epsilon^2\int_0^\infty\diffd
t\,e^{-\epsilon^2t}F(\epsilon,t)\qquad\text{with }F(\epsilon,t)=\int_{t}^\infty\diffd
t'\,f(\epsilon,t').
\label{68}
\end{equation}
Using \autoref{lem:a}, the function $\epsilon\mapsto F(\epsilon,t)$ is
analytic for all $t\ge0$ and $\epsilon\in U$. Furthermore, for some~$\tilde C$,
\begin{equation}
|F(\epsilon,t)|\le \tilde C \min
\left(1,\frac1{t^{\alpha-1+\beta\epsilon}}\right)\qquad\text{for all $t>0$ and
$\epsilon\in U\cap \R$}.
\end{equation}
Then, assuming that the lemma holds for $\alpha-1$, we can apply it to 
the integral with $F$ in \eqref{68}; after Taylor-expanding
the analytic function $F(\epsilon,0)$, we see that the result holds for $f$.
\end{proof}

\section{Expansions in \texorpdfstring{$\epsilon$}{ε}}\label{sec:expansion}

We now finish the proof of \autoref{prop2}. It remains to show \eqref{42}:
\begin{equation}
\int_0^\infty\diffd t \,\big[\varphi(\epsilon,t)-\hat\varphi(\epsilon)\big]e^{-\epsilon^2
t +(1+\epsilon)(\mu_t-2t)}=\mP(\epsilon)+\mO(\epsilon^3),
\label{42bis}
\end{equation}
for some polynomial $\mathcal P(\epsilon)$,
under the hypotheses that $\gamma>1$ and there exists $C>0$, $t_0\ge0$
and a (real) neighbourhood $U$ of 0 such that
\begin{equation}
\big|\varphi(\epsilon,t)-\hat\varphi(\epsilon)\big|\le\frac C t
\qquad\text{for $t>t_0$ and $\epsilon\in U$}.
\label{67}
\end{equation}

Recall from \eqref{BramsonsConditions} that $F(h)\le
C h^{1+p}$ for some $p>0$ and some constant $C$.
Recall the definitions \eqref{defphi} of $\varphi$ and $\hat\varphi$:
\begin{equation}
\varphi(\epsilon,t)   :=\int\diffd z\, F[h(\mu_t+z,t)]e^{(1+\epsilon)z},\qquad
\hat\varphi(\epsilon) :=\int\diffd z\, F[\omega(z)]   e^{(1+\epsilon)z}.
\end{equation}
For each $t>0$, these functions of $\epsilon$ are analytic in the region
$V=\{\epsilon\in \C\,;\,-1<\Re\epsilon <p\}$.
Indeed,
\autoref{lem1} with $r=1$ 
and \eqref{BramsonsConditions}
give that $0\le F[h(\mu_t+z,t)]\le C_{t}
e^{-(1+p)z}$ for some function of time $C_{t}$. Using that bound for $z>0$
and the bound $0\le F[h(\mu_t+z,t)]\le1$ for $z<0$, we get from
\autoref{lem:a} that $\varphi(\epsilon)$ is analytic in the domain
$\{\epsilon\in\C\,;\,-1+a<\Re\epsilon<p-a\}$ for any $a>0$, and is
therefore
analytic on the domain $V$ defined above. The same argument
works in the same way for~$\hat\varphi(\epsilon)$.

Then, as $\gamma>1$, Bramson's result implies that $\mu_t = 2t-\tfrac32\log t +a +o(1)$. With \eqref{67}, we see that there is a constant $C'>0$ such that
\begin{equation}
\big|\varphi(\epsilon,t)-\hat\varphi(\epsilon)\big|e^{
(1+\epsilon)(\mu_t-2t)}
\le \frac {C'}
{t^{\frac52+\frac32\epsilon}}
\qquad\text{for $t>t_0$ and $\epsilon\in U$}.
\label{77}
\end{equation}
Pick $\beta\in(1,\gamma)$, and make the neighbourhood $U$ smaller if needed so that
$\epsilon\in U \implies 0.5 <1+ \epsilon<\beta$.
Then, recalling the definition \eqref{defg} of $g$, 
since $F(h)<h$, and since $e^{(1+\epsilon)z}\le e^{0.5z}+e^{\beta z}$ for
all $z$,
\begin{equation}
\varphi(\epsilon,t)\le g(0.5,t)+g(\beta,t)\qquad\text{for $t>0$ and $\epsilon\in U$},
\end{equation}
which remains bounded for $t\in[0,t_0]$ according to \autoref{lem1}.
Similarly, $\hat\phi(\epsilon)$ is bounded for $\epsilon\in\ U$ and we see
that
the left hand side of \eqref{77} is uniformly bounded by some constant for
$t\le t_0$ and $\epsilon\in U$.
Then, \eqref{42bis} is a direct application of \autoref{lem:O}.
Combined with \eqref{prop21}, this implies that \eqref{prop22} holds:
\begin{equation}\label{prop22bis}
\hat\varphi(\epsilon)\int_0^\infty\diffd t \,e^{-\epsilon^2
t +(1+\epsilon)(\mu_t-2t)}=
\int\diffd x \,h_0(x)e^{(1+\epsilon) x}
-\indic{\epsilon>0}\Phi(2+2\epsilon)+\mP(\epsilon)+\mO(\epsilon^3).
\end{equation}
This concludes the proof of \autoref{prop2}.

We now turn to the proof of \autoref{mainthm}. The first statement (for the
$F(h)=h^2$ case with $h_0$ a compact perturbation of the step function) is
a consequence, with \autoref{prop1}, of the second statement (for a general
$F(h)$). We assume
that the hypotheses of that second statement hold; they imply in particular that the second form
of the magical relation holds, \eqref{prop22} or \eqref{prop22bis}. Furthermore, the hypothesis that $\int
\diffd x\,h_0(x)e^{rx}<\infty$ for some $r>1$ (\textit{i.e.}\@
$\gamma>1$) implies (see \autoref{lem:a}) that $\int \diffd x\,h_0(x)
e^{(1+\epsilon)x}$ is an analytic function of $\epsilon$ around 0,
 so that it can be absorbed into the
$\mP(\epsilon)+\mO(\epsilon^3)$ term. Finally, we write \eqref{prop22bis}
as
\begin{equation}
\label{prop22ter}
\hat\varphi(\epsilon)I(\epsilon)=
-\indic{\epsilon>0}\Phi(2+2\epsilon)+\mP(\epsilon)+\mO(\epsilon^3)
\qquad\text{with }I(\epsilon)=\int_1^\infty\diffd t \,e^{-\epsilon^2
t +(1+\epsilon)(\mu_t-2t)}.
\end{equation}
Notice that we defined $I(\epsilon)$ as an integral from 1 to $\infty$, not
0 to $\infty$. We are allowed to do this because the remaining 
integral from 0 to~1  is an analytic function of $\epsilon$ around~0;
 multiplied by $\hat\varphi(\epsilon)$ (another analytic function), it can be
absorbed into $\mP(\epsilon)+\mO(\epsilon^3)$ term.

Using the lemmas proved in \autoref{sec:lem}, we now compute a small
$\epsilon$ expansion of $I(\epsilon)$. We have assumed that the position $\mu_t$ has
a large $t$ expansion given by
\eqref{position}; actually, let us simply write
\begin{equation}\label{position2}
\mu_t=2t-\frac32\log t +a +\frac{b}{\sqrt t}+c
\frac{\log t}t+r(t)\qquad\text{with }r(t)=\mathcal O\Big(\frac1t\Big),
\end{equation}
and we will recover the values of $b$ and $c$ as given in \eqref{position}.
We have
\begin{equation}
\label{prop22terx}
\begin{aligned}
e^{(1+\epsilon)(\mu_t-2t)}
&=\frac{e^{(1+\epsilon)a}}{t^{\frac32+\frac32\epsilon}}e^{\frac{b(1+ \epsilon)}{\sqrt t}+\frac{c(1+      \epsilon)\log t}{t}+(1+\epsilon) r(t)}
\\
&=e^{(1+\epsilon)a}\Big[\frac1{t^{\frac32+\frac32\epsilon}}+\frac{b(1+\epsilon)}{t^{2
+\frac{3\epsilon}2}}
        +\frac{c(1+\epsilon)\log
t}{t^{\frac52+\frac32\epsilon}}
\Big]
+R(t,\epsilon)
\end{aligned}
\end{equation}
With 
\begin{equation}
R(t,\epsilon)=\frac{e^{(1+\epsilon)a}}{t^{\frac32+\frac32\epsilon}}\Big[e^{\frac{b(1+ \epsilon)}{\sqrt t}+\frac{c(1+ \epsilon)\log
t}{t}+(1+\epsilon) r(t)} - \Big(1+ {\frac{b(1+ \epsilon)}{\sqrt t}+\frac{c(1+ \epsilon)\log
t}{t}}\Big)\Big].
\end{equation}
For $|u|<1$, we have the bound $|e^{u}-(1+u-v)|\le |e^{u}-(1+u)|+|v|\le
u^2+|v|$. Applying this to
$u=\frac{b(1+ \epsilon)}{\sqrt t}+\frac{c(1+ \epsilon)\log
t}{t}+(1+\epsilon)r(t)$ and $v=(1+\epsilon) r(t)$, we see easily that there exists a 
$C>0$ and a real neighbourhood $U$ of 0 such that
\begin{equation}
\label{prop22tery}
|R(t,\epsilon)|\le\frac C {t^{\frac52+\frac32\epsilon}}\qquad\text{for all $t>1$ and all $\epsilon\in
U$.}
\end{equation}
As $\epsilon\mapsto R(t,\epsilon)$ is analytic around 0 for all $t>0$, a direct application of \autoref{lem:1s}, and in particular of
\eqref{lem1seq}, and of \autoref{lem:O} gives from \eqref{prop22ter},
\eqref{prop22terx} and \eqref{prop22tery}:
\begin{equation}
\begin{aligned}
I(\epsilon)&=
e^{(1+\epsilon)a}
\Big[|\epsilon|^{1+3\epsilon}\Gamma(-\tfrac12-\tfrac32\epsilon) 
+b(1+\epsilon)
|\epsilon|^{2+3\epsilon}\Gamma(-1-\tfrac32\epsilon)
\\&\qquad
+c(1+\epsilon)|\epsilon|^{3+3\epsilon}\Big(
-2\log|\epsilon|\,\Gamma(-\tfrac32-\tfrac32\epsilon)
+\Gamma'(-\tfrac32-\tfrac32\epsilon)
\Big)
\Big]
+\mathcal P(\epsilon)+\mathcal O(\epsilon^3).
\end{aligned}
\end{equation}
The three analytic functions from \autoref{lem:1s} have been absorbed
into the $\mathcal P(\epsilon)+\mathcal O(\epsilon)$ term of
\autoref{lem:O}. We expand all the Gamma functions; the expansion of
the second one is irregular:
\begin{equation}
\Gamma(-1-\tfrac32\epsilon)=
\frac{\Gamma(-\tfrac32\epsilon)}{-1-\tfrac32\epsilon}
=\frac{\Gamma(1-\tfrac32\epsilon)}{-\tfrac32\epsilon(-1-\tfrac32\epsilon)}
=
\frac2{3\epsilon}\frac{1+\gamma_E\tfrac32\epsilon+\mO(\epsilon^2)}{1+\tfrac32\epsilon}
=
\frac{2}{3 \epsilon }+\gamma_E-1
+\mathcal O(\epsilon),
\end{equation}
where $\gamma_E=-\Gamma'(1)\simeq0.577$ is Euler's gamma constant.
We obtain

\begin{align}
\label{expI}
I(\epsilon)&=
e^{(1+\epsilon)a}|\epsilon|^{3\epsilon}\Big[\Gamma(-\tfrac12)|\epsilon|-\tfrac32\Gamma'(-\tfrac12)\epsilon|\epsilon|+b\big(\tfrac23\epsilon+(\gamma_E-\tfrac13)\epsilon^2\big)-2c\Gamma(-\tfrac32)|\epsilon|^3\log|\epsilon|\Big]+\mO(\epsilon^3)
\\&=e^{(1+\epsilon)a}|\epsilon|^{3\epsilon}\Big[
 \Big(\Gamma(-\tfrac12)|\epsilon|
+\tfrac23b \epsilon\Big)
+
\Big( b(\gamma_E-\tfrac13)\epsilon^2
-\tfrac32 \Gamma'(-\tfrac12)\epsilon |\epsilon|\Big)
-2c\Gamma(-\tfrac32)|\epsilon|^3\log|\epsilon|
\Big]+\mO(\epsilon^3).\notag
\end{align}
Notice how the expansion mixes terms such as $\epsilon$ and $|\epsilon|$.
For reference, we recall that
\begin{equation}\label{Gamma}
\Gamma(-\tfrac12)=-2\sqrt\pi,\qquad
\Gamma'(-\tfrac12)=-2\sqrt\pi(2-\gamma_E-2\log2),\qquad
\Gamma(-\tfrac32)=\tfrac43\sqrt\pi.
\end{equation}

Before going further, we show how to recover the values of $b$ and $c$.
Notice in \eqref{prop22ter} that, for $\epsilon<0$ (and since
$\hat\varphi(\epsilon)$ is analytic around 0), we must have $I(\epsilon)=\mathcal
P(\epsilon)+\mO(\epsilon)$. In particular, there must remain no
$\log|\epsilon|$ term in the expansion~\eqref{expI} for $\epsilon<0$. 
There is a $\log|\epsilon|$ term explicitly written in \eqref{expI}, and
others in the expansion of the prefactor
$|\epsilon|^{3\epsilon}=1+3\epsilon\log|\epsilon|+\tfrac92\epsilon^2\log^2|\epsilon|+\cdots$.
Developing, we obtain a term $ \big(\Gamma(-\tfrac12)|\epsilon|
+\tfrac23b \epsilon\big)3\epsilon\log|\epsilon|$; that term must
cancel for $\epsilon<0$, hence, with \eqref{Gamma}
\begin{equation}
b=\frac32\Gamma(-\tfrac12)= -3\sqrt\pi.
\label{valb}
\end{equation}
Then, $c$ must be chosen in order to prevent a term
$\epsilon^3\log|\epsilon|$ from appearing when $\epsilon<0$.
This
leads to
\begin{equation}
3\Big(b(\gamma_E-\tfrac13)+\tfrac32\Gamma'(-\tfrac12)\Big)+2c\Gamma(-\tfrac32)=0
\end{equation}
With \eqref{Gamma} and \eqref{valb}, this leads to
\begin{equation}\label{valc}
c=\tfrac98(5-6\log2).
\end{equation}
Using the values \eqref{valb} and \eqref{valc} of $b$ and $c$ in
\eqref{position2} gives back the expression \eqref{position} of the
position $\mu_t$ of the front.
Let us make two remarks:
\begin{itemize}
\item
If we try to add in \eqref{position2} extra terms of the
form $C (\log t)^n / t^\alpha$, we would obtain non-cancellable
singularities (terms containing $\log|\epsilon|$ or non integral powers of
$|\epsilon|$) in the expansion of $I(\epsilon)$. We conclude that if $\mu_t$
can be written as an expansion in terms of the form  $C (\log t)^n
/ t^\alpha$, then the only terms that may appear are those written in
\eqref{position2}.
\item We used the hypothesis 
 that $\int
\diffd x\,h_0(x)e^{rx}<\infty$ for some $r>1$ (\textit{i.e.}\@
$\gamma>1$), only once, to get rid of the $\int\diffd x\,
h_0(x)e^{(1+\epsilon)x}$ term in \eqref{prop22bis}. If we relax this
hypothesis and simply assume $\int\diffd x\,
h_0(x) xe^{x}<\infty$ (this is needed to reach \eqref{prop22bis}), we would
have at this point that, for $\epsilon<0$,
$\hat\varphi(\epsilon)I(\epsilon)=\int\diffd x\,
h_0(x)e^{(1+\epsilon)x}+\mP(\epsilon)+\mO(\epsilon^3)$.
We have just shown that, if the position $\mu_t$ of the front is given by
\eqref{position}, then $I(\epsilon)=\mP(\epsilon)+\mO(\epsilon^3)$ for
$\epsilon<0$. We thus see that 
\begin{equation}
\mu_t\text{ given by \eqref{position}}\implies
\int\diffd x\,
h_0(x)e^{(1+\epsilon)x}=\mP(\epsilon)+\mO(\epsilon^3)\text{ for
$\epsilon<0$}\iff \int\diffd x\, h_0(x) x^3e^{x}<\infty.
\end{equation}
(We omit the proof of the last equivalence.) 
Conversely, if $\int\diffd x\, h_0(x) x^3e^{x}=\infty$, then the
asymptotic expansion for small negative $\epsilon$ of $\int\diffd x\,
h_0(x)e^{(1+\epsilon)x}$ will feature some singular terms larger than
$\epsilon^3$, and the expression of $\mu_t$ needs to be modified in such
a way that $\hat\varphi(\epsilon)I(\epsilon)$ matches those singular terms.
\end{itemize}

We return to the expression \eqref{expI} of $I(\epsilon)$ without
making any assumption on the sign of
$\epsilon$, and we make the substitution
\begin{equation}
|\epsilon|=-\epsilon+2\epsilon\indic{\epsilon>0},\qquad
|\epsilon|^3=-\epsilon^3+2\epsilon^3\indic{\epsilon>0}.
\end{equation}
We have tuned $b$ and $c$ so that one obtains
$I(\epsilon)=\mP(\epsilon)+\mO(\epsilon)$ for $\epsilon<0$. For $\epsilon$
of either sign,
we have three extra terms multiplied by $\indic{\epsilon>0}$,
corresponding to the three terms with $|\epsilon|$ or $|\epsilon|^3$ in
\eqref{expI}:
\begin{align}
I(\epsilon)&=
\indic{\epsilon>0}e^{(1+\epsilon)a}\epsilon^{3\epsilon}\Big[
2\Gamma(-\tfrac12)\epsilon
-3 \Gamma'(-\tfrac12)\epsilon^2
-4c\Gamma(-\tfrac32)\epsilon^3\log\epsilon
\Big]
+\mP(\epsilon)+\mO(\epsilon^3).
\end{align}
Comparing with \eqref{prop22ter}, we see that
we must have (only for $\epsilon>0$, of course):
\begin{equation}
\Phi(2+2\epsilon)=\hat\varphi(\epsilon)e^{(1+\epsilon)a}\epsilon^{3\epsilon}\Big[
-2\Gamma(-\tfrac12)\epsilon
+3 \Gamma'(-\tfrac12)\epsilon^2
+4c\Gamma(-\tfrac32)\epsilon^3\log\epsilon
\Big]+\mO(\epsilon^3).
\label{Phi1}
\end{equation}
This expression will be, after some transformations, our main result \eqref{Phi}. We now make a small
$\epsilon$ expansion of $\hat\varphi(\epsilon)$.
From the definition~\eqref{defphi} of $\hat\varphi(\epsilon)$ and
the equation \eqref{propomega} followed by $\omega$, one has
\begin{equation}
\hat\varphi(\epsilon)=\int\diffd z\,F[\omega(z)]e^{(1+\epsilon)z}=
\int\diffd z\,\big[\omega''(z)+2\omega'(z)+\omega(z)\big]e^{(1+\epsilon)z}
\end{equation}
This function $\hat\phi(\epsilon)$ is analytic around $\epsilon=0$, but we need to assume
$-1<\epsilon<0$ to split the integral into
three terms and integrate by parts. (Recall that $\omega(z)\sim\tilde\alpha z e^{-z}$ as
$z\to\infty$.)
\begin{equation}
\begin{aligned}
\hat\varphi(\epsilon)
&=
 \int\diffd z\,\omega''(z)e^{(1+\epsilon)z}+
2\int\diffd z\,\omega'(z) e^{(1+\epsilon)z}+
 \int\diffd z\,\omega(z)e^{(1+\epsilon)z},
\\&=
\big[(1+\epsilon)^2-2(1+\epsilon)+1\big]
\int\diffd z\,\omega(z) e^{(1+\epsilon)z}=
\epsilon^2\int\diffd z\,\omega(z) e^{(1+\epsilon)z},
\\&= \epsilon^2 e^{-(1+\epsilon)\alpha} \int\diffd z\,\omega(z-a)
e^{(1+\epsilon)z}\qquad\text{for $-1<\epsilon<0$.}
\end{aligned}
\end{equation}
Recall
\eqref{omegaalphabeta}: for any $q\in(0,p)$,
\begin{equation}
\omega(z-a)= (\alpha z + \beta ) e^{-z}+\mathcal O(e^{-(1+q)z})\quad\text{as
$z\to\infty$},
\end{equation}
Then $\int\diffd z\,\omega(z-\alpha)e^{(1+\epsilon)z} = \alpha/\epsilon^2-\beta/\epsilon+\mO(1)$ and
\begin{equation}
\hat\varphi(\epsilon)e^{(1+\epsilon)a}= \alpha-\beta\epsilon+\mO(\epsilon^2).
\end{equation}
Even though the intermediate steps are only valid for $\epsilon<0$, the
final result is also valid for $\epsilon>0$ (small enough) by analyticity.
In \eqref{Phi1}, after replacing $c$ and the Gamma functions by their
values \eqref{Gamma} and \eqref{valc}, we obtain
\begin{equation}
\begin{aligned}
\Phi(2+2\epsilon)
&=(\alpha-\beta\epsilon)\epsilon^{3\epsilon}\Big[
4\sqrt\pi\epsilon
-6\sqrt\pi(2-\gamma_E-2\log2)\epsilon^2
+6(5-6\log2)\sqrt\pi\epsilon^3\log\epsilon
\Big]+\mO(\epsilon^3)
\\
&=\sqrt\pi(\alpha-\beta\epsilon)\epsilon^{3\epsilon}\Big[
4\epsilon
-6(2-\gamma_E-2\log2)\epsilon^2
+6(5-6\log2)\epsilon^3\log\epsilon
\Big]+\mO(\epsilon^3)
\end{aligned}
\end{equation}
It remains to develop with the term
$\epsilon^{3\epsilon}=1+3\epsilon\log\epsilon+\frac92\epsilon^2\log^2\epsilon+\cdots$;
only the coefficient of $\epsilon^3\log\epsilon$ requires to combine two
terms: $3\times (-6)(2-\gamma_E-2\log2)+6(5-6\log 2)=6(3\gamma_E-1)$. We obtain.
\begin{equation}
\Phi(2+2\epsilon)=\sqrt\pi(\alpha-\beta\epsilon)
\Big[
4\epsilon+12\epsilon^2\log\epsilon
-6(2-\gamma_E-2\log2)\epsilon^2
+18\epsilon^3\log^2\epsilon
+6(3\gamma_E-1)\epsilon^3\log\epsilon
\Big]+\mO(\epsilon^3)
\end{equation}

The last step is to replace $\epsilon$ by $\epsilon/2$
\begin{equation}
\begin{aligned}
\Phi(2+\epsilon)&=\sqrt\pi\Big(\alpha-\frac\beta2\epsilon\Big)
\Big[ 2\epsilon+3\epsilon^2\log\frac\epsilon2
-3\Big(1-\frac{\gamma_E}2-\log2\Big)\epsilon^2
+\frac94\epsilon^3\log^2\frac\epsilon2
\\&\qquad\qquad\qquad\qquad\qquad\qquad
+\frac34(3\gamma_E-1)\epsilon^3\log\frac\epsilon2
\Big]+\mO(\epsilon^3)
\\&=\sqrt\pi\Big(\alpha-\frac\beta2\epsilon\Big)
\Big[ 2\epsilon+3\epsilon^2\log\epsilon
-3\Big(1-\frac{\gamma_E}2\Big)\epsilon^2
+\frac94\epsilon^3\log^2\epsilon
\\&\qquad\qquad\qquad\qquad\qquad\qquad
+\frac34(3\gamma_E-6\log2-1)\epsilon^3\log\epsilon
\Big]+\mO(\epsilon^3),
\end{aligned}
\end{equation}
which is \eqref{mainresult}. This completes the proof of the
second part of \autoref{mainthm}. It now remains to prove \autoref{prop1t}
to obtain the first part of \autoref{mainthm}.

\section{Proof of \autoref{prop1t}}\label{proofprop1t}

We start by recalling the main results of \cite{Graham.2019}:
\begin{thm}[Cole Graham 2019 \cite{Graham.2019}]\label{thmCole}
Let $h(x,t)$ be the solution to the Fisher-KPP equation~\eqref{FKPP} with
$F(h)=h^2$ and with
initial condition $h_0(x)$.
Assume that $0\le h_0\le 1$ and that $h_0$ is a compact perturbation of the
step function. There exist $\alpha_0$ and $\alpha_1$ in $\mathbb R$
depending on the initial data $h_0$ such that the following holds. For any
$\gamma>0$, there exists $C_\gamma>0$ also depending on $h_0$ such that for
all $x\in\R$ and all $t\ge3$
\begin{equation}
\big|h(\sigma_t + x,t) -U_\text{app}(x,t)\big|\le \frac{C_\gamma (1+|x|)
e^{-x}}{t^{\frac32-\gamma}},
\label{Cole1}
\end{equation}
where 
\begin{equation}
\sigma_t = 2t -\frac32\log t+\alpha_0 -\frac{3\sqrt\pi}{\sqrt t}
+\frac98(5-6\log 2)\frac{\log t} t +\frac{\alpha_1}t
\end{equation}
and
\begin{equation}
U_\text{app}(x,t)=\phi(x)+\frac1t \psi(x)+\mathcal
O(t^{\gamma-3/2})\text{\quad locally uniformly in $x$}.
\end{equation}
Here, $\phi(x)$ is the critical travelling wave translated in such a way
that \cite[eqn.~(1.2)]{Graham.2019}
\begin{equation}
\phi(x)=A_0 x e^{-x} + \mathcal O(e^{-(1+q)x})\quad\text{as $x\to\infty$},
\label{phiasymp}
\end{equation}
and 
$\psi(x)$ satisfies \cite[Lem.~5 with
$\psi(x)=A_0e^{-x}V_1^-(x)$ as written in the Proof of Thm~3 p.~1985]{Graham.2019}
\begin{equation}
\psi(x)\text{ is bounded},\qquad \psi(x)\sim -\frac{A_0}4 x^3
e^{-x}\quad\text{as $x\to\infty$}.
\label{psiinfty}
\end{equation}

Furthermore, there exist smooth
functions $V^+_1$, $V^+_2$ and $V^+_3$
of $x/\sqrt t$ such that, for $t$ large enough and $\gamma\in(0,3)$,
\begin{equation}
\begin{cases}\displaystyle \Big|U_\text{app}(x,t)-\phi(x)-\frac1t\psi(x)\Big| \le
C_\gamma\frac{\min(1,e^{-x})}{t^{\frac32-\frac23\gamma}} & \text{for $x\le t^{\gamma/6}$},
\\[2ex]\displaystyle \Big| U_\text{app}(x,t) - A_0 e^{-x}\Big(xe^{-x^2/(4t)}+
V^+_1(x/\sqrt t) \\
\qquad\qquad\qquad\qquad+\frac{\log t}{\sqrt t}V^+_2(x/\sqrt
t) +\frac 1 {\sqrt t} V^+_3(x/\sqrt t)\Big)\Big|\le C_\gamma
\dfrac{e^{-x}}{t^{\frac32-\frac12\gamma}}
& \text{for $x>t^{\gamma/6}$}.
\end{cases}
\label{notincole}
\end{equation}
The $V^+_i$ satisfy $V^+_i(0)=V^+_i(\infty)=0$, and so there are
bounded.
\end{thm}
\noindent\textbf{Remarks}\begin{itemize}
\item We introduced in \eqref{propomega} the critical travelling wave $\omega(x)$,
fixing the translational invariance by imposing $\omega(0)=\frac12$. The
functions $\omega$ and $\phi$ are related by
$\omega(x)=\phi\big(\phi^{-1}(\frac12)+x\big)$.
\item \eqref{notincole} is not explicitly written in \cite{Graham.2019},
but it can be pieced together from the proofs:
in the Proof of Theorem~3 p.\,1985, one reads
$U_\text{app}(x,t)=A_0e^{-x}V_\text{app}(x,t)$ and in the proof of
Theorem~9, p.\,1986, one reads
\begin{equation}
V_\text{app}(x,t)=\indic{x<t^\epsilon} V^-(x,t)+\indic{x\ge t^\epsilon}
V^+(x,t)+ K(t) \theta(x t^{-\epsilon}) \varphi(x,t).
\label{V+app}
\end{equation}
At the end of proof (p.\,1995), the author takes $\epsilon=\gamma/6$;
the functions $\varphi$ and $\theta$ are bounded, $K(t)=\mathcal
O(t^{3\epsilon-3/2})$, and $\theta$ is supported on $(0,2)$, see p.\,1986.
Then, we have so far, for some $C$,
\begin{equation}
\begin{cases}\displaystyle
\big|U_\text{app}(x,t)-A_0e^{-x}V^-(x,t)\big| \le
C\frac{e^{-x}\indic{x>0} }{t^{\frac32-\frac12\gamma}} & \text{for $x\le t^{\gamma/6}$},
\\[2ex]\displaystyle \big| U_\text{app}(x,t) - A_0 e^{-x}V^+(x,t)\big|\le C
\frac{e^{-x}}{t^{\frac32-\frac12\gamma}}
& \text{for $x>t^{\gamma/6}$}.
\end{cases}
\label{132}
\end{equation}
($C$ is some positive constant independent of $x$ and $t$ which can change
at each occurrence.)

We start with the first line;
the function $V^-$ is given at the top of p.\,1986:
\begin{equation}
V^-(x,t)=V_0^-(x+\zeta_t)+\frac1t V_1^-(x+\zeta_t),
\end{equation}
with
\begin{equation}
V_0^-(x)=A_0^{-1}e^x\phi(x),\qquad
V_1^-(x)=A_0^{-1}e^x\psi(x),\qquad \zeta(t)=\mathcal O(t^{4\epsilon-\frac32})
=\mathcal O(t^{\frac23\gamma-\frac32}).
\label{134}
\end{equation}
(See respectively p.\,1972, proof of Theorem~3 p.\,1985, and bottom of p.\,1985.)

From p.\,1972 and Lemma~5 p.\,1973, we have 
 $(V_0^-)'(x)\sim 1$ and $(V_1^-)'(x)\sim
-\frac34x^2$ as $x\to\infty$, and $(V_0^-)'(x)=\mathcal O(e^x)$ and
$(V_1^-)'(x)=\mathcal O(e^x)$ as $x\to-\infty$. Thus, 
$|(V_0^-)'(x)|$ and $|(V_1^-)'(x)/t|$ are both bounded by $C\min(e^x,1)$ for all
$t>1$ and all $x<\sqrt t$.
This implies that
\begin{equation}
\Big|V^-(x,t) - V_0^-(x)-\frac1t V_1^-(x)\Big|\le
C \min(e^x,1)
\zeta_t\le C\frac{\min(e^{x},1)}{ t^{\frac32-\frac23\gamma}}
\quad\text{for $t>1$ and $x<\sqrt t$}.
\end{equation}
Multiplying by $A_0e^{-x}$ and using \eqref{134},
\begin{equation}
\Big|A_0e^{-x}V^-(x,t) - \phi(x)-\frac1t \psi(x)\Big|\le
C\frac{\min(1,e^{-x})}{ t^{\frac32-\frac23\gamma}}
\quad\text{for $t>1$ and $x<\sqrt t$}.
\label{85}
\end{equation}
Combining with the first line of \eqref{132} under the assumption $\gamma<3$, we obtain the first line of
\eqref{notincole}, as the bounding term in \eqref{132} is small compared to
the bounding term in \eqref{85}.

We now turn to the second line of \eqref{132}.
The function $V^+$, only defined for $x>0$, is given
in~(3.4) p.\,1973 in terms of $\tau=\log t$ and
$\eta=x/\sqrt t$:
\begin{equation}
V^+(x,t)=e^{\tau/2}V_0^+(\eta)+V_1^+(\eta)+\tau
e^{-\tau/2}V_2^+(\eta)+e^{-\tau/2}V_3^+(\eta),
\label{137}
\end{equation}
with
\begin{equation}
V_0^+(\eta)=\eta e^{-\eta^2/4},\qquad\text{and so}\qquad
e^{\tau/2}V_0^+(\eta)=x e^{-x^2/(4t)}.
\label{138}
\end{equation}
(Top of p.\,1974: $V_0^+(\eta)=q_0\phi_0(\eta)$ for some
real $q_0$; middle of p.\,1974: $q_0=1$; bottom of p.\,1973:
$\phi_0(\eta)=\eta e^{-\eta^2/4}$.)
Using \eqref{137} and \eqref{138} in the second line of \eqref{132}
gives the second line of
\eqref{notincole}.
The $V_i^+$ are smooth (they are solutions on some differential equations
written pp.\,1974,~1975), and satisfy $V_i^+(0)=V_i^+(\infty)=0$, see
line after (3.4) p.\,1973. 
\end{itemize}
We wrote \eqref{notincole} with the accuracy provided by the proofs of
\cite{Graham.2019}, but we actually need a less precise version, only up to
order $1/t$:
\begin{corr}\label{corrCole}
With the notations and hypotheses of \autoref{thmCole},
for any $\gamma\in(0,1/2]$, if $t$ is large enough,
\begin{equation}
\begin{cases}
\displaystyle
\displaystyle\big|h(\sigma_t+x,t)-\phi(x)\big|\le C_\gamma\frac{(1+|x|^3) e^{-x}}t
 &\text{for $x\le t^{\gamma/6}$},
\\[2ex]
\big|h(\sigma_t+x,t)-\phi(x)\big|\le C_\gamma x e^{-x}
&\text{for $x> t^{\gamma/6}$}.
\end{cases}
\label{secondrange}
\end{equation}
\end{corr}
\begin{proof}
Recall from~\eqref{psiinfty} that $\psi$ is bounded and $\psi(x)
\sim C x^3 e^{-x}$ as $x\to\infty$. This implies that
$|\psi(x)|\le
C \min\big(1,(1+|x|^3)e^{-x}\big)$ for some constant $C$. Then, the first line
of \eqref{notincole} implies that
\begin{equation}
\big|U_\text{app}(x,t)-\phi(x)\big| \le
C \frac{\min\big(1,(1+|x|^3)e^{-x}\big)}t \qquad\text{for $x\le
t^{\gamma/6}$}
\label{simple<}
\end{equation}
for some other constant $C$.
With \eqref{Cole1}, this implies the first line of \eqref{secondrange}.
(Recall $\gamma\le\frac12$.)

In the second line of \eqref{notincole}, the quantities $V_i^+$ are
bounded. As $x>t^{\gamma/6}\ge1$, we have
\begin{equation}
\big|U_\text{app}(x,t)\big|\le C x e^{-x}
\qquad\text{for $x> t^{\gamma/6}$}.
\end{equation}
As we also have $\phi(x)\sim A_0 x e^{-x}$, we obtain
$
\big|U_\text{app}(x,t)-\phi(x)\big|\le Cx e^{-x}
\ \text{for $x> t^{\gamma/6}$}.
$
which gives, with \eqref{Cole1}, the second line of
\eqref{secondrange}.
\end{proof}
Unfortunately, \autoref{thmCole} and, consequently,
\autoref{corrCole}
are very imprecise for $x<0$. We will need the following result to
complement \autoref{corrCole}:
\begin{lem}\label{lem4}
With the notations and hypotheses of Theorem 2, there exists $C$
and $t_0$
depending on the initial condition $h_0$ such that, for $t\ge t_0$, 
\begin{equation}
\big|h(\sigma_t+x,t)-\phi(x)\big|\le \frac{C}t
 \qquad\text{for $x\le 0$}.
\label{firstrange}
\end{equation}
\end{lem}
\begin{proof}
Choose $\alpha\in(0,\frac12)$ and  let $x_0=\phi^{-1}(\frac12+\alpha)$.
It suffices to prove $|h(\sigma_t+x,t)-\phi(x)|\le C/t$ for $x\le x_0$:
if $x_0\ge 0$, then \eqref{firstrange} follows; if $x_0<0$, then
\eqref{secondrange} provides the required  bound for $x\in[x_0,0]$.

Let 
\begin{equation}
\delta(x,t)=h(\sigma_t +x ,t)-\phi(x).
\end{equation}
By substitution, one obtains
\begin{equation}\begin{aligned}
\partial_t \delta
=\frac{\diffd}{\diffd t} h(\sigma_t+x,t)
&=\partial_x^2(\phi+\delta)+\dot\sigma_t \partial_x (\phi+\delta)+
(\phi+\delta)-(\phi+\delta)^2,
\\&=\phi''+\dot\sigma_t \phi' +\phi-\phi^2+
\partial_x^2\delta+\dot\sigma_t\partial_x \delta
+\delta-2\delta\phi-\delta^2,
\\&=(\dot\sigma_t-2)\phi'+\partial_x^2\delta+\dot\sigma_t\partial_x
\delta+(1-2\phi -\delta)\delta,
\end{aligned}
\label{eqnr}
\end{equation}
where we used in the last step that $\phi''+2\phi'+\phi-\phi^2=0$.

As $\delta(x,t)$ converges uniformly to 0 \cite{Bramson.1983}, there is
a time $t_0>0$
such that $|\delta(x,t)|\le \alpha$ for all $x$ and all $t\ge t_0$.
Recall that
$\phi(x_0)=\frac12+\alpha$ and $\phi\searrow$. Then
\begin{equation}
1-2\phi(x)+|\delta(x,t)|\le1-2\phi(x_0)+\alpha=-\alpha\qquad\text{for
$x\le x_0$ and $t\ge t_0$}.
\label{c1}
\end{equation}
From respectively \eqref{secondrange} and $|\delta(x,t_0)|\le\alpha$, one can find
$C>0$ such that
\begin{equation}
|\delta(x_0,t)| \le \frac C t   \quad\text{for $t\ge t_0$},\qquad
|\delta(x,t_0)|\le \frac C {t_0}\quad\text{for $x\le x_0$}.
\label{c2}
\end{equation}
As $\phi'<0$ is bounded and
$0<2-\dot\sigma_t \sim \frac3{2t}$ for $t$ large enough, one can
increase $t_0$ and $C$ such that, furthermore,
\begin{equation}
\label{c3}
0\le \phi'(x)(\dot\sigma_t-2) \le \alpha\frac C t -\frac
C {t^2}\quad\text{for $t\ge t_0$ and $x\le x_0$}.
\end{equation}
(The reason for the negligible $C/t^2$ term will soon become apparent.)
Let $\hat \delta$ be the solution to
\begin{equation}
\partial_t \hat \delta =\alpha\frac C t-\frac C{t^2}+\partial_x^2\hat
\delta
+\dot\sigma_t \partial_x \hat \delta-\alpha \delta\quad\text{for $x<x_0$,
$t>t_0$},\qquad \hat \delta(x_0,t)=\frac
C t, \qquad \hat \delta (x,t_0)=\frac C {t_0}.
\label{rhat}
\end{equation}
We consider \eqref{eqnr} for $x<x_0$ and $t>t_0$, taking as
``initial'' condition $\delta(x,t_0)$ and as boundary condition
$\delta(x_0,t)$. 
Using the comparison principle between $\delta$ and $\hat \delta$, and then between
$-\delta$ and $\hat \delta$, one obtains with \eqref{c1},
\eqref{c2} and \eqref{c3} that
$|\delta(x,t)|\le \hat \delta(x,t)$ for all
$x\le x_0$ and $t\ge t_0$.
But the solution to \eqref{rhat} is $\hat \delta(x,t)=\frac C t$, hence $|\delta(x,t)|\le\frac C t$
for $t\ge t_0$ and $x\le x_0$. 
\end{proof}

We can now prove \autoref{prop1t}.

\begin{proof}[Proof of \autoref{prop1t}]
The fact that \eqref{position} holds is already proved in
\cite[Corr.~4]{Graham.2019}, as an easy corollary of \autoref{thmCole},
which states:
\begin{equation}
\mu_t = \sigma_t + \phi^{-1}(\tfrac12)+\mathcal O(\tfrac1t),
\label{musigma}
\end{equation}
so that $a$ in \eqref{position} is given by
$a=\alpha_0+\phi^{-1}(\tfrac12)$.
It remains to prove that \eqref{technical} with $F(h)=h^2$ holds:
\begin{equation}
\left |
\int \diffd x \, e^{rx} h(\mu_t+x,t)^2
-
\int \diffd x \, e^{rx} \omega(x)^2
\right| \le \frac C t
\qquad \text{for $t>t_0$ and $r\in U$},
\label{goal1}
\end{equation}
where $C>0$ and $t_0>0$ are some constants, and where $U$ is some real
neighbourhood of $U$. We choose to take
 $U=[0,01,1.99]$. 

In \eqref{goal1}, make the change of variable $x\to x+\sigma_t-\mu_t$ in
the first integral, and the change $x\to x- \phi^{-1}(\frac12)$ in the
second. Recalling that $\omega(x- \phi^{-1}(\frac12))=\phi(x)$ and
factorizing by $e^{r(\sigma_t-\mu_t)}$, we obtain that \eqref{goal1} is equivalent to
\begin{equation}
e^{r(\sigma_t-\mu_t)}\left |
\int \diffd x \, e^{rx} h(\sigma_t+x,t)^2
-
e^{r(\mu_t-\sigma_t-\phi^{-1}(\frac12))}
\int \diffd x \, e^{rx} \phi(x)^2
\right| \le \frac C t
\quad \text{for $t>t_0$ and $r\in U$}.
\end{equation}
The prefactor $e^{r(\sigma_t-\mu_t)}$ is bounded for $r\in U$ and $t>1$,
and can be dropped. As $\int \diffd x\, e^{rx}\phi(x)^2$ is bounded for
$r\in U$, and since \eqref{musigma} holds, one has for some $C$ and $t_0$:
\begin{equation}
\left|
\int \diffd x \, e^{rx} \phi(x)^2
-
e^{r(\mu_t-\sigma_t-\phi^{-1}(\frac12))} 
\int \diffd x \, e^{rx} \phi(x)^2
\right| \le \frac C t
\quad \text{for $t>t_0$ and $r\in U$},
\end{equation}
and then \eqref{goal1} is equivalent to
\begin{equation}
\left |
\int \diffd x \, e^{rx} h(\sigma_t+x,t)^2
-
\int \diffd x \, e^{rx} \phi(x)^2
\right| \le \frac C t
\qquad \text{for $t>t_0$ and $r\in U$}.
\label{goal2}
\end{equation}
We now show that \eqref{goal2} holds.

First notice that there exists $C>0$ such that, for all $x$ and all
$t$ large enough,
\begin{equation}
h(\sigma_t+x,t)\le C\phi(x).
\label{hphi}
\end{equation}
Indeed, 
from  \eqref{secondrange}, $|h(\sigma_t+x,t)-\phi(x)|\le 2C_\gamma xe^{-x}$ for
$x\ge1$ and $t$ large enough (we used $x^2\le t$ in the first line, since
$\gamma\le1/2$). Since $\phi(x)\sim A_0xe^{-x}$ for large $x$, this implies
that $h(\sigma_t+x,t)\le C \phi(x)$ for some $C$ is $x\ge1$ and $t$ large
enough. Making $C$ larger if needed so that $C\phi(1)\ge1$ ensures that the
relation also holds for $x\le1$ since $\phi\searrow$ and $h\le1$.

Then, for another constant $C$, for all $t$ large enough and all $r\in U$,
\begin{equation}\begin{aligned}
\left |
\int\diffd x\, e^{rx} \Big[ h(\sigma+x,t)^2-\phi(x)^2\Big]
\right |,
&\le
\int\diffd x\, e^{rx} \Big|
h(\sigma+x,t)-\phi(x)\Big|\times\Big(h(\sigma+x,t)+\phi(x)\Big)
  \\&
\le 
C \int\diffd x\, e^{rx} \phi(x) \Big| h(\sigma+x,t)-\phi(x)\Big|.
\end{aligned} \end{equation}
We cut the integral in three ranges: $x<0$,
$0<x<t^{\gamma/6}$ and $x>t^{\gamma/6}$. In the first range, we use
$r\ge0.01$, $\phi\le1$ and \eqref{firstrange}. In the two other ranges, we
use $r\le1.99$ and \eqref{secondrange}:
\begin{equation}
\begin{aligned}
\left|\int\diffd x\, e^{rx} \Big[ h(\sigma+x,t)^2-\phi(x)^2\Big]\right|
&\le C\int_{-\infty}^0\diffd x\,e^{0.01x}\frac 1 t
+  C\int_0^{t^{\gamma/6}}\diffd x\,e^{1.99x}\phi(x)\frac {(1+x^3)e^{-x}}t
\\&\qquad\qquad
+ C \int_{t^{\gamma/6}}^\infty\diffd x\,e^{1.99x}\phi(x)   xe^{-x},
\\&
\le \frac C t+ \frac C t + C t^{\gamma/3}e^{-0.01t^{\gamma/6}} \le \frac C t,
\end{aligned}
\end{equation}
where we used $e^{1.99x}\phi(x)   xe^{-x}\le C x^2e^{-0.01x}$.
This concludes the proof.
\end{proof}

\section{Conclusion}
In this paper, we study the quantity $\Phi(c)$ appearing in \eqref{Phi},
which describes the behaviour of the solution to the Fisher-KPP equation
\eqref{FKPP} at
time $t$ and position $ct$ for $c>2$. We first showed that \eqref{Phi}
holds for values of $c>2$ satisfying \eqref{Phih0}, and we computed a small $\epsilon=c-2$ expansion of the quantity
$\Phi(c)$ appearing in \eqref{Phi}, up to the order $\mO(\epsilon^3)$, see
\eqref{mainresult}.
The expansion depends on the initial condition and the non-linear term in
\eqref{FKPP}
through two numbers
$\alpha$ and $\beta$ which characterize the shifted travelling wave reached
by the front, see \eqref{Bramson's result} and \eqref{omegaalphabeta}.
Although, $\Phi'(2)$ exists, $\Phi''(2)$ does not. The expansion
\eqref{mainresult} is surprisingly irregular, with several logarithmic corrections. 

Our method to reach this result relies on so-called magical relation between the
position $\mu_t$ of the front, the initial condition $h_0$, and the
quantity $\Phi(c)$, see \autoref{prop2}. This approach relates in some way
 the large $t$ expansion \eqref{position} of the position $\mu_t$ of the front
and the small $\epsilon$ expansion of $\Phi(2+\epsilon)$.

As explained in the proofs of the present paper and in
\cite{BerestyckiBrunetDerrida.2018}, the magical relation also allows to
predict non-rigorously the coefficients of the large $t$ expansion of the position of the
front for all initial conditions. It would be interesting to turn this
approach into a proof.

We believe that our result is universal; however, the proofs in this paper
rely on knowing the large $t$ expansion of the position of the front, and
on some other technical condition \eqref{technical} which has only been
proved for the Fisher-KPP equation \eqref{FKPP} with the $F(h)=h^2$
non-linearity, and an initial condition which is a compact perturbation of
the step function. Therefore, our result is only proved in that situation.

All the results in this paper could be easily extended to the front studied
in \cite{BerestyckiBrunetDerrida.2017,
BerestyckiBrunetDerrida.2018,BerestyckiBrunetPenington.2019}, where the
non-linearity in the Fisher-KPP equation is replaced by a moving boundary:
$\partial_th=\partial_x^2h+h$ if $x>\mu_t$ and $h(x,t)=1$ if $x\le\mu_t$
with $h$ differentiable at $x=\mu_t$. Then, as can be shown rigorously,
the magical relation \eqref{prop21} still holds with
$\phi(\epsilon,t)=\hat\phi(\epsilon)=1/(1+\epsilon)$, and we believe that
\eqref{mainresult} also holds; the only result missing  to prove it with our
method is that the large $t$ expansion of $\mu_t$ is also given by
\eqref{position} for that model. (The technical condition \eqref{technical}
is not needed in that case.)

The magical relation could also be used to compute $\Phi(c)$ for large $c$.
As is clear from inspecting \eqref{prop21}, this would require studying the
early times of the evolution of the front. This point
was already noticed in \cite{DerridaMeersonSasorov.2016}.

Beyond the results themselves, the method used to reach them are,
in our opinion, quite unexpected and interesting. We feel that there
remains many aspects of the Fisher-KPP equation that could be better
understood, and the magical relation might be a useful tool to that
purpose.
\section*{Thanks}
The author wishes to thank Pr.\@ Julien Berestycki for invaluable
discussions.

\appendix\section{Appendix}\label{appendix}
We prove \autoref{lem2}:
\setcounter{lem}{1}
\begin{lem}
Let $\beta\in(0,\gamma)$. For $t>0$, the quantities $h(x,t)$, $|\partial_x
h(x,t)|$, $|\partial_x^2h(x,t)|$ and $|\partial_t h(x,t)|$ are bounded by
$A(t)\max(1,e^{-\beta x})$ for some locally bounded function $A$.
\end{lem}
\begin{proof}
We already know that the result holds for $h(x,t)$ from
\autoref{lem1} and $0<h(x,t)<1$. The Fisher-KPP equation \eqref{FKPP}
will then provide the result for $\partial_t h$ once it is proved for
$\partial_x^2 h$. We now focus on $\partial_x h$ and $\partial_x^2h$.
Following Uchiyama \cite[section~4]{Uchiyama.1978}, we
use the following representations:
\begin{align}
h(x,t) &= \int\diffd y \,p(x-y,t) h_0(y) + \int_0^t\diffd s\int\diffd y\,
p(x-y,t-s) f[h(y,s)],
\label{U1}
\\
\partial_x h(x,t) &= \int\diffd y\,\partial_x p(x-y,t) h_0(y) + \int_0^t\diffd s\int\diffd y\,
\partial_x p(x-y,t-s) f[h(y,s)],
\label{U2}\\
\partial_x^2 h(x,t) &= \int\diffd y\,\partial_x^2 p(x-y,t) h_0(y) + \int_0^t\diffd s\int\diffd y\,
\partial_x p(x-y,t-s) f'[h(y,s)]\partial_x h(y,s),
\label{U3}
\end{align}
where $f(h)=h-F(h)$ and
\begin{equation}
p(x,t)=\frac1{\sqrt{4\pi t}}e^{-\frac{x^2}{4t}}.
\end{equation}
By using $0\le h_0\le1$  and $0\le f[h]\le h\le 1$ in \eqref{U2},
Uchiyama shows that
\begin{equation}
\label{U4}
|\partial_x h(x,t)| \le 
\int\diffd y\, |\partial_x p(x-y,t)|+
\int_0^t\diffd s\int\diffd y\, |\partial_x p(x-y,t-s)|
=
\frac1{\sqrt\pi}\left(\frac1{\sqrt t}+2\sqrt t\right),
\end{equation}
Let $\zeta:=\max | f' |$; by using \eqref{U4} (and $0\le h_0\le1$) in
\eqref{U3}, he also obtains, in the same way,
\begin{equation}
\label{U5}
|\partial_x^2 h(x,t)| \le \frac1{2t}+\zeta\times(1+t).
\end{equation}

We now need to show that $|\partial_x h(x,t)|$ and $|\partial_x^2 h(x,t)|$
are bounded by $A(t)e^{-\beta x}$ for some locally bounded function $A$ and
for $\beta\in(0,\gamma)$.

We bound 
the right-hand-sides of  \eqref{U2} and \eqref{U3}, starting with the terms
involving $h_0$.
Choose $p>1$ such that $p\beta<\gamma$, and let $q$ be the Hölder conjugate
of $p$, \textit{i.e.}\@ such that $1/p+1/q=1$. By Hölder's inequality
applied to $h_0(y)^{1/p}e^{\beta y} \times \partial_xp(x-y,t) h_0(y)^{1/q}
e^{-\beta y}$, we obtain
\begin{equation}
\left|\int\diffd y\,\partial_x p(x-y,t) h_0(y)\right|\le
\left[\int\diffd y\, h_0(y)e^{\beta p y}\right]^{\frac1p}
\left[\int\diffd y\, |\partial_xp(x-y,t)|^q h_0(y) e^{-\beta
qy}\right]^{\frac1q}.
\label{holder1}
\end{equation}
The first integral in the right-hand-side is $g(\beta p,0)$, which is
finite since we took $\beta p<\gamma$. We focus on the second integral,
which we first bound using $h_0(y)\le 1$; then
\begin{equation}
\begin{aligned}
\int\diffd y\, |\partial_xp(x-y,t)|^q e^{-\beta qy}
&= e^{-\beta q x}\int\diffd y\, |\partial_xp(y,t)|^q e^{\beta qy},
\\&= e^{-\beta q x}\sqrt t \int\diffd y\, |\partial_xp(y\sqrt t,t)|^q 
e^{\beta qy\sqrt t},
\\&= e^{-\beta q x}t^{-q+\frac12} \int\diffd y\, |\partial_xp(y,1)|^q 
e^{\beta qy\sqrt t}.
\end{aligned}
\label{45}
\end{equation}
Indeed, as $p(y\sqrt t,t)=p(y ,1) /\sqrt t$, we have that
$\partial_x p (y\sqrt t,t) = \partial_x p (y,1)/t$. The remaining
integral on the right-hand-side converges because of the Gaussian bounds in
$\partial_x p(y,1)$ and gives some continuous function of $t$ defined for
all $t\ge0$. Then, in \eqref{holder1},
\begin{equation}\label{46}
\left|\int\diffd y\,\partial_x p(x-y,t) h_0(y)\right|\le
\frac{B_1(t)}{t^{1-\frac1{2q}}}e^{-\beta x},
\end{equation}
for some function $B_1$ continuous on $[0,\infty)$.
It is crucial for what follows that $B_1(0)$ is finite, so that the
divergence of $B_1(t)/t^{1-1/({2q})}$ as $t\searrow0$ is
integrable.

Note: we are about to introduce functions $B_2$, $B_3$, etc. As for $B_1$,
all these functions are implicitly defined and continuous on $[0,\infty)$.

The same method
for the second derivative, using this time
$\partial_x^2 p (y\sqrt t,t) = \partial_x^2 p (y,1)/t^{3/2}$,  gives
\begin{equation}
\left|\int\diffd y\,\partial_x^2 p(x-y,t) h_0(y)\right|\le
\frac{B_2(t)}{t^{\frac32-\frac1{2q}}}e^{-\beta x}.
\label{47}
\end{equation}

We now turn to the second term in the right hand side of \eqref{U2}.
We first write, from \autoref{lem1},
\begin{equation}
0\le f[h(y,s)]\le h(y,s)
\le C\frac{e^{(1+\beta^2)s}}{\sqrt s} e^{-\beta y}
\le C\frac{e^{(1+\beta^2)t}}{\sqrt s} e^{-\beta y},\label{48}
\end{equation}
for $0<s\le t$, with $C$ a constant.
Then, as in \eqref{45},
\begin{equation}\begin{aligned}
\int\diffd y\, |\partial_x p(x-y,t-s)| e^{-\beta y}
&=e^{-\beta x} \int\diffd y\, |\partial_x p(y,t-s)| e^{\beta y}
=\frac{e^{-\beta x}}{\sqrt{t-s}} \int\diffd y\,|\partial_x p(y,1)|e^{\beta
y \sqrt {t-s}}
\\&\le
\frac{e^{-\beta x}}{\sqrt{t-s}} \int\diffd y\,|\partial_x p(y,1)|e^{\beta
\max(y,0) \sqrt {t}}
= \frac{B_3(t)}{\sqrt{t-s}}e^{-\beta x},
\end{aligned}
\label{49}
\end{equation}
for $0\le s < t$.
This leads with \eqref{48} to
\begin{equation}
\left|\int\diffd y\, \partial_x p(x-y,t-s) f[h(y,s)]\right|\le
\frac{B_4(t)}{\sqrt{s(t-s)}}e^{-\beta x},
\end{equation}
for $0<s<t$. Then,
\begin{equation}
\left|\int_0^t\diffd s\int\diffd y\, \partial_x p(x-y,t-s) f[h(y,s)]\right|\le
\pi B_4(t) e^{-\beta x}.
\end{equation}
With \eqref{46}, into \eqref{U2}, we obtain
\begin{equation}
\big|\partial_xh(x,t)\big|\le
\frac{B_5(t)}{t^{1-\frac1{2q}}}e^{-\beta x}.
\label{52}
\end{equation}

We now turn to the second term in the right hand side of \eqref{U3}. We
bound $f'[h(y,s)]$ by $\zeta:=\max |f'|$; then using \eqref{52} with
\eqref{49}, we obtain
\begin{equation}
\left|\int\diffd y\,
\partial_x p(x-y,t-s) f'[h(y,s)]\partial_x h(y,s)\right|\le \zeta 
\frac{B_3(t)B_5(s)}{\sqrt{t-s}\,s^{1-\frac1{2q}}}e^{-\beta x},
\end{equation}
and, finally, since the integral on $s$ is finite,
\begin{equation}
\left|\int_0^t\diffd s\int\diffd y\,
\partial_x p(x-y,t-s) f'[h(y,s) )\partial_x h(y,s)\right|\le B_6(t)
e^{-\beta x}.
\end{equation}
With \eqref{47}, into \eqref{U3}:
\begin{equation}
\big|\partial_x^2h(x,t)\big|\le
\frac{B_7(t)}{t^{\frac32-\frac1{2q}}}e^{-\beta x},
\end{equation}
and the proof is complete.
\end{proof}

\printbibliography
\end{document}